\documentclass[11pt]{article}
\usepackage{amssymb}	% Allows use of AMS's mathematical symbols
\usepackage{amsmath}    % Allows use of AMS's mathematical commands
\usepackage{epsfig}
\usepackage{latexsym}
\usepackage{multirow}
\usepackage{longtable}
\usepackage{graphicx}
\allowdisplaybreaks

\newenvironment{proof}[1][Proof:]{\begin{trivlist} 
\item[\hskip \labelsep {\bfseries #1}]}{\end{trivlist}} 
\newcommand{\qed}{\nobreak \ifvmode \relax \else \ifdim\lastskip<1.5em \hskip-\lastskip \hskip1.5em plus0em minus0.5em \fi \nobreak \vrule height0.75em width0.5em depth0.25em\fi}

\def\R{{\bf R}}

\def\Tr{{\rm T}}
\def\T{{\rm T}}

\def\diag{{\rm diag}}

\newtheorem{algorithm}{Algorithm}[section]

\newtheorem{theorem}{Theorem}[section]
\newtheorem{lemma}{Lemma}[section]
\newtheorem{proposition}{Proposition}[section]
\newtheorem{corollary}{Corollary}[section]
\newtheorem{remark}{Remark}[section]

\usepackage{color}
\usepackage{ulem}

\begin{document}
\title{Two computationally efficient polynomial-iteration infeasible interior-point algorithms for linear programming
}
\author{Y. Yang\footnote{\normalsize Office of Research, US NRC, Two White Flint North 11545 Rockville Pike, Rockville, MD 20852-2738, United States. Email: yaguang.yang@verizon.net.} \\
%Communicated by T. Rapcsak
}
\date{\today}

\maketitle    % This command generates the title.

\begin{abstract}
Linear programming has been one of the most extensively studied 
branches in mathematics and has been found many applications in
science and engineering.
Since the beginning of the development of interior-point methods, 
there exists a puzzling gap between the results in theory and the 
observations in numerical experience, i.e., algorithms with good
polynomial bound are not computationally efficient and algorithms
demonstrated efficiency in computation do not have a good or any 
polynomial bound. Todd raised a question in 2002: ``Can we find a 
theoretically and practically efficient way to reoptimize?'' This 
paper is an effort to close the gap. We propose two arc-search 
infeasible interior-point algorithms with infeasible central path
neighborhood wider than all existing infeasible interior-point
algorithms that are proved to be convergent. We show that the 
first algorithm is polynomial and its simplified version has a 
complexity bound equal to the best known
complexity bound for all (feasible or infeasible) interior-point
algorithms. We demonstrate the computational efficiency of the 
proposed algorithms by testing all Netlib linear programming 
problems in standard form and comparing the numerical results 
to those obtained by Mehrotra's predictor-corrector algorithm 
and a recently developed more efficient arc-search algorithm 
(the convergence of these two algorithms is unknown). We 
conclude that the newly proposed algorithms are not only 
polynomial but also computationally competitive comparing
to both Mehrotra's predictor-corrector algorithm and 
the efficient arc-search algorithm.
 
\end{abstract}

{\bf Keywords:} Polynomial algorithm, arc-search, infeasible interior-point method, linear programming.

{\bf AMS subject classifications:} 90C05, 90C51.
\newpage
 
\section{Introduction}

Interior-point method has been regarded as a mature technique of linear programming 
since the middle of 1990s \cite[page 2]{wright97} . However, there still exist several obvious gaps between 
the results in theory and the observations in computational experience. First, algorithms 
using proven techniques have inferior polynomial bound. For example, higher-order 
algorithms that use second or higher derivatives have been proved to improve the 
computational efficiency \cite{Mehrotra92,lms91,lms92} but higher-order algorithms 
have either poorer polynomial bound than the first-order algorithms \cite{mar90} or 
do not even have a polynomial bound \cite{Mehrotra92}. Second, some algorithm 
with the best polynomial bound performs poorly in real computational test. For example, 
short step interior-point algorithm, which searches optimizer in a narrow neighborhood,
has the best polynomial bound \cite{kmy89a} but 
performs very poorly in practical computation, while the long step interior-point 
algorithm \cite{kmy89b}, which searches optimizer in a larger neighborhood,
performs much better in numerical test but has inferior 
polynomial bound \cite{wright97}. Even worse, Mehrotra's predictor-corrector (MPC) algorithm, which has been widely 
regarded as the most efficient interior-poin algorithm in 
computation and is competitive to the simplex algorithm for 
large problems, has not been proved to be polynomial (it may not 
even be convergent \cite{cartis09}). It should be noted that 
lack of polynomiality was a serious concern for simplex method \cite{kleeMinty72}, 
therefore motivated ellipsoid method for linear programming \cite{Khachiyan79},
and was one of the main arguments in the early development of interior-point algorithms 
\cite{wright97,karmarkar84}. Because of these dilemmas, Todd asked in 2002 \cite{todd02} ``can we 
give a theoretical explanation for the difference between worst-case bounds and 
observed practical performance?  Can we find a theoretically and practically efficient 
way to reoptimize?''

In several recent papers, we tried to close these gaps. In \cite{yang13}, we 
proposed an arc-search interior-point algorithm for linear programming which uses 
higher-order derivatives to construct an ellipse to approximate the central path. 
Intuitively, searching along this ellipse will generate a longer step size than 
searching along any straight line. Indeed, we showed that the arc-search (higher-order) algorithm 
has the best polynomial bound which has partially solved the first dilemma. 
We extended the method and proved a similar result for convex quadratic programming \cite{yang11}; 
and we demonstrated some promising numerical test results. 

%In a different direction, we proposed an improved version of the %short step algorithm 
%\cite{yang13a} that {\it optimally selects centering parameter and the step size at the 
%same time}. By using this strategy, we showed that not only the %short-step algorithm 
%achieves better polynomial bound but also performs more efficient in practical computation
%than a long step algorithm, which has partially addressed second %dilemma and shows
%the importance of using a two-pronged approach for both centering parameter and step size. 

The algorithms proposed in \cite{yang13,yang11} assume that 
the starting point is feasible and the central path does exist. Unfortunately, 
most Netlib test problems (and real problems) do not meet these assumptions.
As a matter of the fact, most Netlib test problems do not even have an interior-point
as noted in \cite{cg06}. To demonstrate the superiority of the arc-search
strategy proposed in \cite{yang13,yang11} for practical problems, 
we devised an arc-search infeasible interior-point algorithm in \cite{yang16}, which 
allows us to test a lot of more Netlib problems. The proposed algorithm is very similar to  
Mehrotra's algorithm but replaces search direction by an arc-path suggested
in \cite{yang13,yang11}. The comprehensive numerical test 
for a larger pool of Netlib problems reported in \cite{yang16} clearly shows 
the superiority of arc-search over traditional line-search method. 

Because the purpose of \cite{yang16} is to demonstrate the 
computational merit of the arc-search method, the algorithm in 
\cite{yang16} is a mimic of Mehrotra's algorithm and we have not 
shown its convergence\footnote{In fact, we notice in \cite{yang16}
that Mehrotra's algorithm does not converge for several Netlib
problems.}. The purpose of this paper is to develop some infeasible 
interior-point algorithms which are both computationally competitive to 
simplex algorithms (i.e., at least as good as Mehrotra's algorithm) 
and theoretically polynomial. We first devise an algorithm slightly 
different from the one in \cite{yang16} and we show that this algorithm 
is polynomial. We then propose a simplified version of the first
algorithm. This simplified algorithm will search optimizers in a 
neighborhood larger than those used in short step and long step 
path-following algorithms, thereby generating potentially a 
larger step size. Yet, we want to show that the modified 
algorithm has the best polynomial 
complexity bound, in particular, we want to show that the
complexity bound is better than ${\mathcal O}(n^2L)$ in 
\cite{zhang94} which was established for an infeasible 
interior-point algorithm searching optimizers in the long-step 
central-path neighborhood, and than ${\mathcal O}(nL)$ in 
\cite{mizuno94,jmiao06,yy18} which were obtained for infeasible 
interior-point algorithms using a narrow short-step central-path 
neighborhoods. As a matter of fact, the simplified algorithm 
achieves the complexity bound 
${\mathcal O}(\sqrt{n}L)$, which is the same as the best 
polynomial bound for feasible interior-point algorithms.

To make these algorithms attractive in theory and efficient
in numerical computation, we remove some unrealistic and 
unnecessary assumption made by existing infeasible interior-point 
algorithms for proving some convergence results. First, we 
do not assume that the initial point has the 
form of $(\zeta e, 0, \zeta e)$, where $e$ is a vector of all
ones, and $\zeta$ is a scalar which is 
an upper bound of the optimal solution as defined by
$\| (x^*,s^*) \|_{\infty} \le \zeta$, which is an unknown before 
an optimal solution is found. Computationally, some most efficient 
staring point selections in \cite{Mehrotra92,lms92} do not meet this 
restriction. Second, we remove a requirement\footnote{In the
requirement of $\| (r_b^k, r_c^k) \| \le [\| (r_b^0, r_c^0) \| 
/\mu_0] \beta \mu_k$,  $\beta \ge 1$ is a constant, $k$ is the 
iteration count, $r_b^k$ and $r_c^k$, are residuals of primal 
and dual constraints respectively, and $\mu_k$ is the duality
measure. All of these notations will be given in Section 2.} that
$\| (r_b^k, r_c^k) \| \le [\| (r_b^0, r_c^0) \| /\mu_0] \beta \mu_k$, 
which is required in existing convergence analysis for infeasible
interior-point algorithms, for example,
equation (6.2) of \cite{wright97}, equation (13) of 
\cite{kojima96}, and Theorem 2.1 of \cite{mizuno94}. Our extensive numerical experience 
shows that this unnecessary requirement is the major barrier 
to achieve a large step size for the existing infeasible interior-point 
methods. This is the main reason that existing infeasible 
interior-point algorithms with proven convergence results do not 
perform well comparing to Mehrotra's algorithm which
does not have any convergence result. We demonstrate the 
computational merits of the proposed algorithms by testing 
these algorithms along with Mehrotra's algorithm and the 
very efficient arc-search algorithm in \cite{yang16} 
using Netlib problems and comparing the test results. To have
a fair comparison, we use the same initial point, the same
pre-process and post-process, and the same termination 
criteria for the four algorithms in all test problems.

The reminder of the paper is organized as follows. Section 2 is a
brief problem description. Section 3 describes the main ideas of 
the arc-search algorithms. Section 4 presents the first algorithm 
and proves its convergence. Section 5 presents the second 
algorithm, which is a simplified version of the first algorithm, 
and discusses the condition of the convergence. 
Section~\ref{implSec} provides the implementation details. 
Section 7 reports and compares the numerical test results obtained 
by the proposed algorithms, the popular Mehrotra's algorithm, and 
the recently developed efficient arc-search algorithm in 
\cite{yang16}. The conclusions are summarized in the last section.

\section{Problem Descriptions}

%This section is similar to many literatures. We present it here
%simply for completeness.
Consider the Linear Programming in the standard form:
\begin{eqnarray}
\min \hspace{0.05in} c^{\T}x, \hspace{0.15in} \mbox{\rm subject to} 
\hspace{0.1in}  Ax=b, \hspace{0.1in} x \ge 0,
\label{LP}
\end{eqnarray}
where $A \in {\bf R}^{m \times n}$, $b \in {\bf R}^{m} $, $c \in {\bf R}^{n}$ 
are given, and $x \in {\bf R}^n$  is the vector to be optimized. Associated 
with the linear programming is the dual programming that is also presented in the standard form:
\begin{eqnarray}
\max \hspace{0.05in} b^{\T}\lambda, \hspace{0.15in} \mbox{\rm subject to} 
\hspace{0.1in}  A^{\T}\lambda+s=c, \hspace{0.1in} s \ge 0,
\label{DP}
\end{eqnarray}
where dual variable vector $\lambda \in {\bf R}^{m}$, and dual slack vector 
$s \in {\bf R}^{n}$. 

Interior-point algorithms require all the iterates satisfying the
conditions $x > 0$ and $s > 0$. Infeasible interior-point 
algorithms, however, allow the iterates deviating from the
equality constraints. 
Throughout the paper, we will denote the residuals of the equality constraints 
(the deviation from the feasibility) by
\begin{equation}
r_b=Ax-b, \hspace{0.1in} r_c=A^{\T}\lambda + s-c,
\label{residuals}
\end{equation}
the duality measure by 
\begin{equation}
\mu=\frac{x^{\T}s}{n},
\label{duality}
\end{equation}
the identity matrix of any dimension by $I$, the vector of all 
ones with appropriate dimension by $e$, the Hadamard (element-wise) 
product of two vectors $x$ and $s$ by $x \circ s$. To make the notation 
simple for block column vectors, we will denote, for example, a point in 
the primal-dual problem $[x^{\T}, \lambda^{\T}, s^{\T}]^{\T}$ 
by $(x,\lambda, s)$. We will denote the initial point (a vector) 
of any algorithm by $(x^0,\lambda^0, s^0)$, the corresponding duality 
measure (a scalar) by $\mu_0$, the point after the $k$th iteration by 
$(x^k,\lambda^k, s^k)$, the corresponding duality measure by $\mu_k$, the optimizer by $(x^*, \lambda^*, s^*)$, and the corresponding duality measure by $\mu_*$. For $x \in \R^n$, we will denote the $i$th component of $x$ by $x_i$, the Euclidean 
norm of $x$ by $\| x \|$, a related diagonal matrix by 
$X \in \R^{n \times n}$ whose diagonal elements are the 
components of $x$. Finally, we denote by $\emptyset$ the empty set.

The central-path ${\mathcal C}$ of the primal-dual linear programming problem
is parameterized by a scalar $\tau >0$ as follows. For each interior 
point $(x, \lambda, s) \in {\mathcal C}$ on the central path, there 
is a $\tau >0$ such that 
\begin{subequations}
\begin{align}
Ax=b  \label{patha} \\
A^{\T}\lambda+s=c \label{pathb} \\
(x,s) >0 \label{pathd}  \\
x_is_i = \tau, \hspace{0.1in} i=1,\ldots,n \label{pathc}.
\end{align}
\label{centralpath}
\end{subequations}
\noindent
Therefore, the central path is an arc in ${\bf R}^{2n+m}$ parameterized as a 
function of $\tau$ and is denoted as 
\begin{equation}
{\mathcal C} = \lbrace(x(\tau), \lambda(\tau), s(\tau)): \tau >0 \rbrace.
\end{equation} 
As $\tau \rightarrow 0$, the central path $(x(\tau), \lambda(\tau), 
s(\tau))$ represented by (\ref{centralpath}) approaches to a solution of 
the linear programming problem represented by (\ref{LP}) because 
(\ref{centralpath}) reduces to the KKT condition as $\tau \rightarrow 0$. 

Because of the high cost of finding the initial feasible point and its corresponding central-path 
described in (\ref{centralpath}), we consider a modified problem which allows infeasible iterates
on an arc ${\mathcal A}(\tau)$ which satisfies the following conditions.
\begin{subequations}
\begin{align}
Ax(\tau)-b =\tau r_b^k:=r_b^k(\tau), \label{arcaa} \\
A^{\T}\lambda(\tau)+s(\tau)-c = \tau r_c^k:=r_c^k(\tau), \label{arcbb} \\
(x(\tau),s(\tau)) > 0, \label{arcdd}  \\
x(\tau) \circ s(\tau) = \tau x^k \circ s^k.  \label{arcc}
\end{align}
\label{neiborArc}
\end{subequations}
where $x(1)=x^k$, $s(1)=s^k$, $\lambda(1)=\lambda^k$, $r_b(1)=r_b^k$, $r_c(1)=r_c^k$,
$(r_b (\tau), r_c (\tau)) = \tau (r_b^k, r_c^k)  \rightarrow 0$ as $\tau \rightarrow 0$. 
Clearly, as $\tau \rightarrow 0$, the arc defined as above
will approach to an optimal solution of (\ref{LP}) because 
(\ref{neiborArc}) reduces to KKT condition as $\tau \rightarrow 0$.
We restrict the search for the optimizer in either the 
neighborhood ${\mathcal F}_1$ or the neighborhood ${\mathcal F}_2$ 
defined as follows: 
\begin{subequations}
\begin{gather}
{\mathcal F}_1=\lbrace(x, \lambda, s): \hspace{0.1in} (x,s) >0,
\hspace{0.1in} x_i^ks_i^k \ge \theta \mu_k \rbrace, 
\label{infeasible1} \\
{\mathcal F}_2=\lbrace(x, \lambda, s): \hspace{0.1in} (x,s) >0 \rbrace,
\label{infeasible2}
\end{gather}
\end{subequations}
where $\theta \in (0,1)$ is a constant. The neighborhood 
(\ref{infeasible2}) is clearly the widest neighborhood used in all existing 
literatures. Throughout the paper, we make the following assumption. 
\newline
\newline
{\bf Assumption 1:}
\begin{itemize}
\item[] {\it $A$ is a full rank matrix. }
\end{itemize}

Assumption 1 is trivial and non-essential as $A$ can always be 
reduced to meet this condition in polynomial operations. With this 
assumption, however, it will significantly simplify the 
mathematical treatment. In Section~\ref{implSec}, we will describe 
a method based on \cite{andersen95} to check if a problem 
meets this assumption. If it is not, the method will reduce 
the problem to meet this assumption. 
\newline
\newline
{\bf Assumption 2:}
\begin{itemize}
\item[] {\it There is at least an optimal solution of (\ref{LP}), 
i.e., the KKT condition holds. }
\end{itemize}

This assumption implies that there is at least one feasible 
solution of (\ref{LP}), which will be used in our convergence
analysis.

\section{Arc-Search for Linear Programming}
\label{arcSearch}

Although the infeasible central path defined in (\ref{neiborArc}) 
allows for infeasible initial point, the calculation of (\ref{neiborArc})
is still not practical. We consider a simple approximation of 
(\ref{neiborArc}) proposed in \cite{yang16}.
Starting from any point $(x^k, \lambda^k, s^k)$ with 
$(x^k, s^k)>0$, we consider a specific arc which is defined by 
the current iterate and $(\dot{x}, \dot{\lambda}, \dot{s})$ and 
$(\ddot{x}, \ddot{\lambda}, \ddot{s})$ as follows:
\begin{equation}
\left[
\begin{array}{ccc}
A & 0 & 0\\
0 & A^{\T} & I \\
S^k & 0 & X^k
\end{array}
\right]
\left[
\begin{array}{c}
\dot{{x}} \\ \dot{\lambda}  \\  \dot{{s}}
\end{array}
\right]
=\left[
\begin{array}{c}
r_b^k \\ r_c^k \\ {x^k} \circ {s^k} 
\end{array}
\right],
\label{doty}
\end{equation}
\begin{equation}
\left[
\begin{array}{ccc}
A & 0 & 0\\
0 & A^{\T} & I \\
S^k & 0 & X^k
\end{array}
\right]
\left[
\begin{array}{c}
\ddot{{x}}(\sigma_k) \\ \ddot{\lambda}(\sigma_k)  \\  \ddot{{s}}(\sigma_k)
\end{array}
\right]
=\left[
\begin{array}{c}
0 \\ 0 \\ -2\dot{x} \circ \dot{s} + \sigma_k \mu_k e
\end{array}
\right],
\label{ddoty}
\end{equation}
where $\sigma_k \in [0,1]$ is the centering parameter introduced 
in \cite[page 196]{wright97}, and the duality measure $\mu_k$ 
is evaluated at $(x^k, \lambda^k, s^k)$. We emphasize that 
the second derivatives are functions of $\sigma_k$ which 
we will carefully select in the range of $0< \sigma_{\min} \le \sigma_k \le 1$.
A crucial but normally not emphasized fact is that if $A$ is full rank, 
and $X^k$ and $S^k$ are positive diagonal matrices, then the matrix
\[
\left[
\begin{array}{ccc}
A & 0 & 0\\
0 & A^{\T} & I \\
S^k & 0 & X^k
\end{array}
\right]
\]
is nonsingular. This guarantees that (\ref{doty}) and (\ref{ddoty})
have unique solutions, which is important not only in theory but 
also in computation, to all interior-point algorithms in linear 
programming. Therefore, we introduce the following assumption.
\newline
\newline
{\bf Assumption 3:}
\begin{itemize}
\item[] {\it $X^k > 0$ and $S^k> 0$ are bounded below from
zeros for all $k$ iterations until the program is terminated. }
\end{itemize}
A simple trick of rescaling the step-length introduced by Mehrotra
\cite{Mehrotra92} (see our implementation in Section 6.10) 
guarantees that this assumption holds. 

Given the first and second derivatives defined by (\ref{doty}) and (\ref{ddoty}),
an analytic expression of an ellipse, which is an approximation 
of the curve defined by (\ref{neiborArc}),  is derived in 
\cite{yang09,yang13}.  

\begin{theorem}
Let $(x(\alpha),\lambda(\alpha),s(\alpha))$ be an ellipse 
defined by $(x(\alpha),\lambda(\alpha),s(\alpha)) |_{\alpha=0}
=(x^k,\lambda^k,s^k)$ and its first 
and second derivatives $(\dot{x}, \dot{\lambda}, \dot{s})$ and 
$(\ddot{x}, \ddot{\lambda}, \ddot{s})$ which are defined by
(\ref{doty}) and (\ref{ddoty}). Then the ellipse is an
approximation of ${\mathcal A}(\tau)$ and is given by
\begin{equation}
x(\alpha,\sigma) = x^k - \dot{x}\sin(\alpha)+\ddot{x}(\sigma) (1-\cos(\alpha)).
\end{equation}
\begin{equation}
\lambda(\alpha,\sigma) = \lambda^k-\dot{\lambda}\sin(\alpha)
+\ddot{\lambda}(\sigma) (1-\cos(\alpha)).
\end{equation}
\begin{equation}
s(\alpha,\sigma) = s^k - \dot{s}\sin(\alpha)+\ddot{s}(\sigma) (1-\cos(\alpha)).
\end{equation}
\label{ellipseSX}
%\hfill \qed
\end{theorem}

It is clear from the theorem that the search of optimizer is 
carried out along an arc parameterized by $\alpha$ with
an adjustable parameter $\sigma$. 
Therefore, we name the search method as to arc-search.
We will use several simple results that can easily be derived from (\ref{doty}) and (\ref{ddoty}). 
To simplify the notations, we will drop the superscript and subscript $k$ unless a confusion 
may be introduced. We will also use $(\ddot{x}, \ddot{\lambda}, \ddot{s})$ instead of 
$(\ddot{x}(\sigma), \ddot{\lambda}(\sigma), \ddot{s}(\sigma))$.
Using the relations 
\[
\dot{x}=A^{\T}(AA^{\T})^{-1}r_b + \hat{A} v, \hspace{0.1in}
A^{\T} \dot{\lambda} + \dot{s} = r_c, \hspace{0.1in}
X^{-1}\dot{x}+S^{-1} \dot{s}=e,
\]
and using the similar derivation of \cite[Lemma 3.5]{yang13} or \cite[Lemma 3.3]{yang13a}, we have 
\begin{eqnarray}
& & X^{-1}\dot{x}= X^{-1}[A^{\T}(AA^{\T})^{-1}r_b + \hat{A} v],
\hspace{0.05in} S^{-1}(A^{\T} \dot{\lambda} +\dot{s}) = S^{-1}r_c,
\hspace{0.05in} X^{-1}\dot{x}+S^{-1} \dot{s}=e
\nonumber \\
 & \iff & \left[ \begin{array}{cc}
X^{-1}\hat{A}, &  -S^{-1}A^{\T} 
\end{array} \right]
\left[ \begin{array}{c}
v \\ \dot{\lambda} 
\end{array} \right]
= e -X^{-1}A^{\T}(AA^{\T})^{-1}r_b -S^{-1}r_c
\nonumber \\
 & \iff & \left[ \begin{array}{c}
v \\ \dot{\lambda} 
\end{array} \right]
= \left[ \begin{array}{c}
( \hat{A}^{\Tr}SX^{-1} \hat{A})^{-1} \hat{A}^{\Tr}S,  
 \\  -(AXS^{-1}A^{\Tr})^{-1}AX 
\end{array} \right]
\left[  e -X^{-1}A^{\T}(AA^{\T})^{-1}r_b -S^{-1}r_c  \right].
\end{eqnarray}
This gives the following analytic solutions for (\ref{doty}).
\begin{subequations}
\begin{align}
\dot{x}= \hat{A} (\hat{A}^{\T}SX^{-1}\hat{A})^{-1}\hat{A}^{\T} S (e-X^{-1}A^{\T}(AA^{\T})^{-1}r_b 
- S^{-1} r_c )+A^{\T}(AA^{\T})^{-1}r_b,   \\
\dot{\lambda}=- ( {A}XS^{-1} {A}^{\T})^{-1} A X (e-X^{-1}A^{\T}(AA^{\T})^{-1}r_b - S^{-1} r_c ),   \\
\dot{s}= {A}^{\T} ( {A}XS^{-1} {A}^{\T})^{-1} A X (e-X^{-1}A^{\T}(AA^{\T})^{-1}r_b - S^{-1} r_c )+r_c,  
\end{align}
\label{py1}
\end{subequations}
These relations can easily be reduced to some simpler formulas 
that will be used in Sections 6.6 and 6.8.
\begin{subequations}
\begin{align} 
 ( {A}XS^{-1} {A}^{\T})\dot{\lambda}= AXS^{-1} r_c - b, \label{pl2} \\
\dot{s}= r_c- {A}^{\T}\dot{\lambda}, \label{ps2} \\
\dot{x}= x-X S^{-1}\dot{s}. \label{px2} 
\end{align}
\label{doy3}
\end{subequations}
Several relations follow immediately from (\ref{doty}) and (\ref{ddoty}) (see also in \cite{yang16}).
\begin{lemma}
Let $(\dot{x}, \dot{\lambda}, \dot{s})$ and $(\ddot{x}, \ddot{\lambda}, \ddot{s})$ be defined in 
(\ref{doty}) and (\ref{ddoty}). Then, the following relations hold.
\begin{equation}
s \circ \dot{x} + x \circ \dot{s} = x \circ s, \hspace{0.1in} 
s^{\T} \dot{x} + x^{\T} \dot{s} = x^{\T} {s},
\hspace{0.1in}
s^{\T} \ddot{x} + x^{\T} \ddot{s} = -2\dot{x}^{\T} \dot{s} +\sigma \mu n,
\hspace{0.1in}
\ddot{x}^{\T} \ddot{s} =0.
\end{equation}
\label{simple}
%\hfill \qed
\end{lemma}

Most popular interior-point algorithms of linear programming
(e.g. Mehrotra's algorithm) use heuristics to select $\sigma$
first and then select the step size. In \cite{yang13a}, it has 
been shown that a better strategy is to select both $\sigma$ and 
step size $\alpha$ at the same time. This requires representing 
$(\ddot{x}, \ddot{\lambda}, \ddot{s})$ explicitly 
in terms of $\sigma_k$. 
Let $\hat{A}$ be the orthonormal base of the null space of $A$. 
Applying the similar derivation of 
\cite[Lemma 3.3]{yang13a} to (\ref{ddoty}), we have the following explicit
solution for $(\ddot{x}, \ddot{\lambda}, \ddot{s})$ in terms of $\sigma$.
\begin{subequations}
\begin{align}
\ddot{x}=\hat{A} (\hat{A}^{\T}SX^{-1}\hat{A})^{-1}\hat{A}^{\T} X^{-1}(-2 \dot{x} \circ \dot{s} + \sigma \mu e)
:= {p}_x\sigma+ {q}_x, \label{ddx2} \\
\ddot{\lambda}= -( {A}XS^{-1} {A}^{\T})^{-1} A S^{-1} (-2 \dot{x} \circ \dot{s} + \sigma \mu e) 
:= {p}_{\lambda}\sigma + {q}_{\lambda}, \label{ddl2} \\
\ddot{s}= {A}^{\T} ( {A}XS^{-1} {A}^{\T})^{-1} A S^{-1} (-2 \dot{x} \circ \dot{s} + \sigma \mu e) 
:= {p}_s\sigma + {q}_s. \label{dds2}
\end{align}
\label{ddoy2}
\end{subequations}
These relations can easily be reduced to some simpler formulas
as follows: 
\begin{subequations}
\begin{align}
 ( {A}XS^{-1} {A}^{\T})\ddot{\lambda}= A S^{-1} (2 \dot{x} \circ \dot{s} - \sigma \mu e),   \label{ddlc2} \\
\ddot{s}= -{A}^{\T} \ddot{\lambda}, \label{ddsc2} \\
\ddot{x}=S^{-1} (\sigma \mu e -X \ddot{s} -2 \dot{x} \circ \dot{s} ),  \label{ddxc2} 
\end{align}
\label{ddoy3}
\end{subequations}
or their equivalence to be used in Sections 6.6 and 6.8:
\begin{subequations}
\begin{gather}
( {A}XS^{-1} {A}^{\T}) p_{\lambda} = -A S^{-1} \mu e, \hspace{0.1in} 
( {A}XS^{-1} {A}^{\T}) q_{\lambda}= 2 A S^{-1} (\dot{x} \circ \dot{s}),
\\
p_s = -{A}^{\T} p_{\lambda}, \hspace{0.1in} q_s= -{A}^{\T} q_{\lambda},
\\
p_x = S^{-1} \mu e - S^{-1}X p_s,   \hspace{0.1in} q_x=  -S^{-1}X q_s -2 S^{-1}(\dot{x} \circ \dot{s}).
\end{gather}
\label{detail}
\end{subequations}
From (\ref{ddoy2}) and (\ref{ddoty}), we also have
\begin{equation}
\left[
\begin{array}{ccc}
A & 0 & 0\\
0 & A^{\T} & I \\
S & 0 & X
\end{array}
\right]
\left[
\begin{array}{c}
{p}_{x} \\ {p}_{\lambda}  \\  {p}_{s}
\end{array}
\right]
=\left[
\begin{array}{c}
0 \\ 0 \\ \mu e
\end{array}
\right],
\label{barp}
\end{equation}
\begin{equation}
\left[
\begin{array}{ccc}
A & 0 & 0\\
0 & A^{\T} & I \\
S & 0 & X
\end{array}
\right]
\left[
\begin{array}{c}
{q}_{x} \\ {q}_{\lambda}  \\  {q}_{s}
\end{array}
\right]
=\left[
\begin{array}{c}
0 \\ 0 \\ - 2 \dot{x} \circ \dot{s}
\end{array}
\right].
\label{barq}
\end{equation}
From these relations, it is straightforward to derive the following
\begin{lemma}
Let $( {p}_x, {p}_{\lambda}, {p}_{s})$ and $( {q}_{x}, {q}_{\lambda}, {q}_{s})$ be defined in 
(\ref{barp}) and (\ref{barq}); $(\dot{x}, \dot{\lambda}, \dot{s})$ be defined in (\ref{doty}). 
Then, for every iteration $k$ ($k$ is omitted for the sake of notational simplicity),
the following relations hold.
\begin{subequations}
\begin{gather}
q_x^{\T}  {p}_s =0, \hspace{0.1in} q_s^{\T}  {p}_{x} =0, \hspace{0.1in} q_x^{\T} {q}_{s} =0, 
\hspace{0.1in} p_{x}^{\T} p_{s} =0, \label{dota} \\
s^{\T}p_x+x^{\T}p_s =n\mu, \hspace{0.1in} s^{\T}q_x+x^{\T}q_s =-2\dot{x}^{\T}\dot{s}, \label{dotb}  \\
s \circ p_x + x \circ p_s = \mu e, \hspace{0.1in}  s \circ q_x + x \circ q_s = -2 \dot{x} \circ \dot{s} \label{dotc}.
\end{gather}
\end{subequations}
\label{simple1}
%\hfill \qed
\end{lemma}

\begin{remark}
Under Assumption 3, a simple but very important 
observation from (\ref{py1}) and (\ref{ddoy2}) is that $\dot{x}$,
$\dot{s}$, $\ddot{x}$, and $\ddot{s}$ are all bounded if 
$r_b$ and $r_c$ are bounded, which we will show later.
\end{remark}

To prove the convergence of the first algorithm, 
the following condition is required in every iteration:
\begin{equation}
x^k \circ s^k \ge \theta \mu_k e
\label{cond4cong}
\end{equation}
where $\theta \in (0,1)$ is a constant.

\begin{remark}
Given $(x^k, \lambda^k, s^k, \dot{x}, \dot{\lambda},\dot{s},\ddot{x}, \ddot{\lambda},\ddot{s})$ 
with $(x^{k}, s^k) > (0,0)$, 
our strategy is to use the relations described in Lemmas \ref{simple} and \ref{simple1} to
find some appropriate $\alpha_k \in (0, \pi/2]$ and $\sigma_k \in [\sigma_{\min},1]$ such that 
\begin{itemize}
\item[1.] $\| (r_b^{k+1}, r_c^{k+1}) \|$ and $\mu_{k+1}$ decrease in every iteration and approach to 
zero as $k \rightarrow 0$.
\item[2.] $(x^{k+1}, s^{k+1}) > (0,0)$.
\item[3.]  $x^{k+1} \circ s^{k+1} \ge \theta \mu_{k+1} e$.
\end{itemize}
\label{keyIdea}
\end{remark}

The next lemma to be used in the discussion is taken from 
\cite[Lemma 3.2]{yang16}.
\begin{lemma}
Let $r_b^k = Ax^k-b$, $r_c^k=A^{\T} \lambda^k +s^k -c$, and 
$\nu_k=\prod_{j=0}^{k-1} (1-\sin(\alpha_j))$. Then, the following 
relations hold.
\begin{subequations}
\begin{align}
r_b^k = r_b^{k-1} (1-\sin(\alpha_{k-1})) = \cdots = r_b^0 \prod_{j=0}^{k-1} (1-\sin(\alpha_j))
= r_b^0 \nu_k, \label{a}
\\
r_c^k = r_c^{k-1} (1-\sin(\alpha_{k-1})) = \cdots = r_c^0 \prod_{j=0}^{k-1} (1-\sin(\alpha_j))
= r_c^0 \nu_k. \label{b}
\end{align}
\label{errorUpdate}
\end{subequations}
In a compact form, (\ref{errorUpdate}) can be rewritten as
\begin{equation}
(r_b^k, r_c^k)= \nu_k (r_b^0, r_c^0).
\label{compact}
\end{equation}
\label{basic}
\end{lemma}

Lemma \ref{basic} indicates clearly that: to reduce $(r_b^k, r_c^k)$ fast, 
we should take the largest possible step size $\alpha_k \rightarrow \pi/2$.
If $\nu_k=0$, then $r_b^k=0=r_c^k$. In this case, the problem has 
a feasible interior-point $(x^k, \lambda^k, s^k)$ and can be 
solved by using some efficient feasible interior-point algorithm. 
Therefore, in the remainder of this 
section and the next section, we use the following
\newline
\newline
{\bf Assumption 4:}
\begin{itemize}
\item[] {\it $\nu_k>0$ for $\forall k \ge 0$. }
\end{itemize}

A rescale of $\alpha_k$ similar to the strategy discussed in 
\cite{wright97} is suggested in Section~\ref{implSec}.10, 
which guarantees that Assumptions 3 and 4
hold in every iteration. To examine the decreasing property of 
$\mu(\sigma_k,{\alpha}_k)$, we need the following result.

\begin{lemma}
Let ${\alpha}_k$ be the step length at $k$th iteration for 
$x(\sigma_k, \alpha_k)$, $s(\sigma_k, \alpha_k)$, and 
$\lambda(\sigma_k, \alpha_k)$ be defined in Theorem 
\ref{ellipseSX}. Then, the updated duality measure after an
iteration from $k$ can be expressed as 
\begin{equation}
\mu_{k+1} := \mu(\sigma_k,{\alpha}_k) =
\frac{1}{n} \left[ a_u(\alpha_k)  \sigma_k + b_u(\alpha_k) \right],
\label{usigmaAlpha}
\end{equation}
where 
\[a_u(\alpha_k)=n \mu_k (1-\cos(\alpha_k)) - (\dot{x}^{\T} {p}_s+\dot{s}^{\T} {p}_x) 
\sin(\alpha_k)(1-\cos(\alpha_k))\] 
and 
\[
b_u(\alpha_k)=n \mu_k (1-\sin(\alpha_k)) - [\dot{x}^{\T} \dot{s} (1-\cos(\alpha_k))^2
+(\dot{s}^{\T}q_x + \dot{x}^{\T} q_s) \sin(\alpha_k)(1-\cos(\alpha_k))]
\]
are coefficients which are functions of $\alpha_k$. 
\label{simple2}
\end{lemma}
\begin{proof} Using (\ref{duality}), (\ref{ddoy2}), and Lemmas \ref{simple} and \ref{simple1}, we have
\begin{eqnarray}
& & n \mu (\sigma_k,{\alpha}_k) \nonumber  \\
& = & \left( x^k-\dot{x}\sin(\alpha_k)+\ddot{x}(1-\cos(\alpha_k)) \right)^{\T} 
\left( s^k - \dot{s}\sin(\alpha_k)+\ddot{s}(1-\cos(\alpha_k)) \right)
\nonumber \\
& = & {x^{k^{\T}} s^k}- \left( x^{k^{\T}} \dot{s} + s^{k^{\T}} \dot{x} \right) \sin(\alpha_k)
+ \left( x^{k^{\T}} \ddot{s}  + s^{k^{\T}} \ddot{x} \right)(1-\cos(\alpha_k))
\nonumber \\ 
& & +  { \dot{x}^{\T}\dot{s} } \sin^2(\alpha_k)
   - \left( \dot{x}^{\T} \ddot{s} + \dot{s}^{\T} \ddot{x} \right) 
\sin(\alpha_k) (1-\cos(\alpha_k))
\nonumber \\
& = & n \mu_k (1- \sin(\alpha_k)) + \left( \sigma_k \mu_k n 
-2\dot{x}^{\T}\dot{s} \right)  (1-\cos(\alpha_k)) 
+   \dot{x}^{\T}\dot{s}  \sin^2(\alpha_k) 
\nonumber \\
& &- \left( \dot{x}^{\T} \ddot{s}  
+ \dot{s}^{\T} \ddot{x} \right) \sin(\alpha_k) (1-\cos(\alpha_k))
\nonumber \\ 
& = & n \mu_k (1- \sin(\alpha_k)) +n \sigma_k \mu_k (1-\cos(\alpha_k)) 
-  \dot{x}^{\T}\dot{s} (1-\cos(\alpha_k))^2
\nonumber \\
& & - \left( \dot{x}^{\T} \ddot{s}  
+ \dot{s}^{\T} \ddot{x} \right) \sin(\alpha_k) (1-\cos(\alpha_k))
\label{diffstepsize} \\ 
& = & \left[ n \mu_k (1-\cos(\alpha_k)) 
  - (\dot{x}^{\T} {p}_s+\dot{s}^{\T} {p}_x) 
\sin(\alpha_k)(1-\cos(\alpha_k)) \right] \sigma_k
\nonumber \\ 
& & + 
n \mu_k (1-\sin(\alpha_k)) - [\dot{x}^{\T} \dot{s} (1-\cos(\alpha_k))^2
+(\dot{s}^{\T}q_x + \dot{x}^{\T} q_s) \sin(\alpha_k)(1-\cos(\alpha_k))]
\nonumber \\ 
& := & a_u(\alpha_k)  \sigma_k + b_u(\alpha_k). \nonumber
\end{eqnarray}
This proves the lemma.
\hfill \qed
\end{proof}
The following Lemma is taken from \cite{yang13}.
\begin{lemma} For $\alpha \in [0, \frac{\pi}{2}]$,
\[
\sin^2(\alpha) \ge 1-\cos(\alpha) \ge \frac{1}{2} \sin^2(\alpha).
\]
\label{sincos}
\end{lemma}

Since Assumption 3 implies that $\mu_k = x^{k^{\Tr}} s^k/n$ is bounded
below from zero, in view of Lemma \ref{sincos}, it is easy to see from
(\ref{diffstepsize}) the following proposition.
\begin{proposition}
For any fixed $\sigma_k$, 
if $\dot{x}$, $\dot{s}$, $\ddot{x}$, and $\ddot{s}$ are
bounded, then there always exist $\alpha_k \in (0,1)$ 
bounded below from zero such that $\mu(\sigma_k,{\alpha}_k)$ 
decreases in every iteration. Moreover, 
$\mu_{k+1} : = \mu (\sigma_k,{\alpha}_k) \rightarrow \mu_k (1-\sin(\alpha_k))$
as $\alpha_k \rightarrow 0$.
\label{muDec}
\end{proposition}

Now, we show that there exists $\alpha_k$ bounded below from
zero such that the requirement 2 of Remark \ref{keyIdea} holds.
Let $\rho \in (0,1)$ be a constant, and
\begin{equation}
\underline{x}^k = \displaystyle\min_i x^k_i, \hspace{0.1in}
\underline{s}^k = \displaystyle \min_j s^k_j.
\label{ijmath}
\end{equation}
Denote $\phi_k$ and $\psi_k$
such that 
\begin{align}
\phi_k = \min \{ \rho \underline{x}^k , \nu_k \}, \hspace{0.1in}
\psi_k = \min \{ \rho \underline{s}^k, \nu_k \}. 
\label{phikpsi}
\end{align}
It is clear that 
\begin{subequations}
\begin{align}
0< \phi_k e \le \rho x^k, \hspace{0.1in} 0< \phi_k e \le \nu_{k}e, 
\label{phi} \\
0< \psi_k e \le  \rho s^k, \hspace{0.1in}0< \psi_k e \le  \nu_{k}e.
\label{psi}
\end{align}
\end{subequations}

Positivity of $x(\sigma_k, \alpha_k)$ and $s(\sigma_k,\alpha_k)$ is
guaranteed if $(x^0,s^0)>0$ and the following conditions hold.
\begin{eqnarray}
x^{k+1} & = & x(\sigma_k, \alpha_k) = x^k-\dot{x} \sin(\alpha_k) +\ddot{x}(1-\cos(\alpha_k)) \nonumber \\
& = & {p}_x (1-\cos(\alpha_k)) \sigma_k +
[x^k-\dot{x} \sin(\alpha_k)+ q_x (1-\cos(\alpha_k))]
\nonumber \\
& := & a_x(\alpha_k) \sigma_k  +b_x(\alpha_k)   \ge \phi_k e. 
\label{posi1}
\end{eqnarray}
\begin{eqnarray}
s^{k+1} & = & s(\sigma, \alpha_k) = s^k-\dot{s} \sin(\alpha_k) +\ddot{x}(1-\cos(\alpha_k)) \nonumber \\
& = &  {p}_s (1-\cos(\alpha_k)) \sigma_k  +
[s^k-\dot{s} \sin(\alpha_k)+ q_s (1-\cos(\alpha_k))]
\nonumber \\
& := & a_s(\alpha_k) \sigma_k  +b_s(\alpha_k)  \ge \psi_k e. 
\label{posi2}
\end{eqnarray}
If $x^{k+1}= x^k-\dot{x} \sin(\alpha_k) +\ddot{x}(1-\cos(\alpha_k)) \ge \rho x^k$ holds, from (\ref{phi}), we have
$x^{k+1}\ge \phi_k e$. Therefore, inequality (\ref{posi1}) will be 
satisfied if
\begin{equation}
(1-\rho ) x^k -\dot{x} \sin(\alpha_k) +\ddot{x}(1-\cos(\alpha_k)) \ge 0,
\label{posi3}
\end{equation} 
which holds for some $\alpha_k>0$ bounded below from zero
because $(1-\rho ) x^k >0$ is bounded below from zero.
Similarly, from (\ref{psi}), inequality (\ref{posi2}) will be satisfied if 
\begin{equation}
(1-\rho ) s^k -\dot{s} \sin(\alpha_k) +\ddot{s}(1-\cos(\alpha_k)) \ge 0,
\label{posi4}
\end{equation}
which holds for some $\alpha_k>0$ bounded below from zero
because $(1-\rho ) s^k >0$ is bounded below from zero. 
We summarize the above discussion as the following proposition.
\begin{proposition}
There exists $\alpha_k>0$ bounded below from zero such that
$(x^{k+1}, s^{k+1})>0$ for all iteration $k$.
\label{secondBound}
\end{proposition}

The next proposition addresses requirement 3 of Remark \ref{keyIdea}.
\begin{proposition}
There exist $\alpha_k$ bounded below 
from zero for all $k$ such at (\ref{cond4cong}) holds. 
\label{thirdBound}
\end{proposition}
\begin{proof}
From (\ref{posi1}) and (\ref{posi2}), since $ x_i^k s_i^k \ge \theta \mu_k$,
we have
\begin{eqnarray}
& & x_i^{k+1} s_i^{k+1} \nonumber \\
& = & [x_i^k -\dot{x}_i \sin(\alpha_k) +\ddot{x}_i (1-\cos(\alpha_k))]
[s_i^k -\dot{s}_i \sin(\alpha_k) +\ddot{s}_i (1-\cos(\alpha_k))]  \nonumber \\
& = &  x_i^k s_i^k -[\dot{x}_i^k s_i^k+x_i^k \dot{s}_i^k]\sin(\alpha_k) 
+ [\ddot{x}_i^k s_i^k+x_i^k \ddot{s}_i^k] (1-\cos(\alpha_k))
  \nonumber \\ 
&  & + \dot{x}_i^k \dot{s}_i^k \sin^2(\alpha_k) 
- [\ddot{x}_i^k \dot{s}_i^k+\dot{x}_i^k \ddot{s}_i^k] \sin(\alpha_k) 
(1-\cos(\alpha_k))+ \ddot{x}_i^k \ddot{s}_i^k(1-\cos(\alpha_k))^2
  \nonumber \\ 
& = &   x_i^k s_i^k (1 -\sin(\alpha_k) ) + \dot{x}_i^k \dot{s}_i^k [\sin^2(\alpha_k) - 2(1-\cos(\alpha_k))]
+\sigma_k \mu_k (1-\cos(\alpha_k)) 
  \nonumber \\ 
&  & - [\ddot{x}_i^k \dot{s}_i^k+\dot{x}_i^k \ddot{s}_i^k] \sin(\alpha_k) 
(1-\cos(\alpha_k))+ \ddot{x}_i^k \ddot{s}_i^k(1-\cos(\alpha_k))^2
  \nonumber \\ 
& = &   x_i^k s_i^k (1 -\sin(\alpha_k) ) - \dot{x}_i^k \dot{s}_i^k (1-\cos(\alpha_k))^2 +\sigma_k \mu_k (1-\cos(\alpha_k)) 
  \nonumber \\ 
&  & - [\ddot{x}_i^k \dot{s}_i^k+\dot{x}_i^k \ddot{s}_i^k] \sin(\alpha_k) 
(1-\cos(\alpha_k))+ \ddot{x}_i^k \ddot{s}_i^k(1-\cos(\alpha_k))^2
  \nonumber \\ 
& \ge &   \theta \mu_k (1 -\sin(\alpha_k) ) 
+\sigma_k \mu_k (1-\cos(\alpha_k)) 
  \nonumber \\ 
&  & - [\ddot{x}_i^k \dot{s}_i^k+\dot{x}_i^k \ddot{s}_i^k] \sin(\alpha_k) 
(1-\cos(\alpha_k))+ (\ddot{x}_i^k \ddot{s}_i^k-\dot{x}_i^k \dot{s}_i^k)
(1-\cos(\alpha_k))^2.
\label{thirdBound1}
\end{eqnarray}
Since Assumption 3 implies (a) $\mu_k$ is bounded below from zero, 
and (b) $\dot{x}$, $\dot{s}$, $\ddot{x}$, and $\ddot{s}$ are all 
bounded, the result follows from Lemma \ref{sincos}.
\hfill \qed
\end{proof}

Let $\sigma_{\min}$ and $\sigma_{\max}$ be constants, and
$0 < \sigma_{\min} < \sigma_{\max} \le 1$. From Propositions 3.1
and 3.2, and Lemma 3.3, we conclude:

\begin{proposition}
For any fixed $\sigma_k$ such that
$\sigma_{\min} \le \sigma_k \le \sigma_{\max}$, 
there is a constant $\delta>0$ related to lower bound of 
$\alpha_k$ such that 
(a) $r_b^{k} -r_b^{k+1} \ge \delta$,
(b) $r_c^{k} -r_c^{k+1} \ge \delta$, 
(c) $\mu_k -\mu_{k+1} \ge \delta$, and
(d) $(x^{k+1},s^{k+1})>0$.
\end{proposition}

\section{Algorithm 1}

This algorithm considers the search in the neighborhood 
(\ref{infeasible1}). Based on the discussion in the previous 
section, we will show in this section that the following 
arc-search infeasible interior-point algorithm is well-defined and
converges in polynomial iterations.

\begin{algorithm} {\ } \\ 
Data: $A$, $b$, $c$.  \hspace{0.1in} {\ } \\
Parameter: $\epsilon \in (0,1)$, $\sigma_{\min} \in (0,1)$, 
$\sigma_{max} \in (0,1)$, $\theta \in (0,1)$,
and $\rho \in (0,1)$. {\ } \\
Initial point: ${\lambda}^0=0$ and $(x^0, s^0)>0$.
%, and $0< \nu_0 < \displaystyle\min_i \{ x_i^0, s_i^0 \}$.
\newline
{\bf for} iteration $k=0,1,2,\ldots$
\begin{itemize}
\item[] Step 0: If ${\| r_b^0 \|} \le {\epsilon}$, ${\| r_c^0 \|} \le {\epsilon}$, 
and $\mu_k \le \epsilon$, stop.
\item[] Step 1: Calculate ${\mu}_k$, $r_b^k$, $r_c^k$, $\dot{\lambda}$, $\dot{s}$, 
$\dot{x}$, $p_x^k$, $p_{\lambda}^k$, $p_s^k$, $q_x^k$, $q_{\lambda}^k$, and $q_s^k$. 
\item[] Step 2: Find some appropriate $\alpha_k \in (0,\pi/2]$
and $\sigma_k \in [\sigma_{\min},\sigma_{\max}]$ to satisfy 
\begin{subequations}
\begin{gather}
x(\sigma_k, \alpha_k) \ge \phi_k e, \hspace{0.05in} 
s(\sigma_k, \alpha_k) \ge \psi_k e, \hspace{0.05in}  
\mu_k > \mu(\sigma_k, \alpha_k),    \hspace{0.05in}
x(\sigma_k, \alpha_k) s(\sigma_k, \alpha_k) \ge \theta \mu(\sigma_k, \alpha_k)e.  
 \nonumber 
\end{gather}
\end{subequations}
\item[] Step 3: Set $(x^{k+1}, \lambda^{k+1}, s^{k+1})
=(x(\sigma_k, \alpha_k), \lambda(\sigma_k, \alpha_k), s(\sigma_k, \alpha_k))$ and 
$\mu_{k+1}=\mu(\sigma_k, \alpha_k)$.
\item[] Step 4: Set $k+1 \rightarrow k$. Go back to Step 1.
\end{itemize}
{\bf end (for)} 
\hfill \qed
\label{mainAlgo}
\end{algorithm}

The algorithm is well defined because of the three propositions
in the previous section, i.e., there is a series of $\alpha_k$
bounded below from zero such that all conditions in Step 2 hold. 
Therefore, a constant $\rho \in (0,1)$ satisfying 
$\rho \ge (1- \sin(\alpha_k))$ does exist for all $k \ge 0$.
Denote
\begin{equation}
\beta_k = \frac{\min \{ \underline{x}^k , \underline{s}^k \}}{\nu_k} \ge 0,
\label{betak}
\end{equation}
and 
\begin{equation}
\beta = \displaystyle\inf_{k} \{ \beta_k \} \ge 0.
\label{beta}
\end{equation}
The next lemma shows that $\beta$ is bounded below from zero.
\begin{lemma}
Assuming that $\rho \in (0,1)$ is a constant and for all 
$k \ge 0$, $\rho \ge (1- \sin(\alpha_k))$. Then, we have 
$\beta \ge \min \{ \underline{x}^0 , \underline{s}^0, 1 \}$.
\label{betaGt0}
\end{lemma} 
\begin{proof}
For $k=0$, $(\underline{x}^0,\underline{s}^0)>0$, and $\nu_0 =1$, therefore, 
$\beta_0 \ge \min \{ \underline{x}^0, \underline{s}^0, 1 \}$ holds. Assuming that 
$\beta_k \ge \min \{ \underline{x}^0, \underline{s}^0, 1 \}$ holds for $k>0$, we
would like to show that $\beta_{k+1} =\min \{ \beta_k, 1 \}$ 
holds for $k+1$. We divide our discussion into three cases.
\begin{itemize}
\item[] Case 1: $\min \{ \underline{x}^{k+1}, \underline{s}^{k+1} \}
= \underline{x}^{k+1} \ge \rho \underline{x}^{k} \ge \underline{x}^{k} 
(1- \sin(\alpha_k))$. Then we have
\[
\beta_{k+1} = \frac{\min \{ \underline{x}^{k+1}, \underline{s}^{k+1} \}}{\nu_{k+1}} 
\ge \frac{ \underline{x}^{k} 
(1- \sin(\alpha_k))}{\nu_k (1- \sin(\alpha_k))} \ge \beta_k. 
\]
\item[] Case 2: $\min \{ \underline{x}^{k+1}, \underline{s}^{k+1} \}
= \underline{s}^{k+1} \ge \rho \underline{s}^{k} \ge \underline{s}^{k} 
(1- \sin(\alpha_k))$. Then we have
\[
\beta_{k+1} = \frac{\min \{ \underline{x}^{k+1}, \underline{s}^{k+1} \}}{\nu_{k+1}} 
\ge \frac{ \underline{s}^{k} 
(1- \sin(\alpha_k))}{\nu_k (1- \sin(\alpha_k))} \ge \beta_k. 
\]
\item[] Case 3: $\min \{ \underline{x}^{k+1}, \underline{s}^{k+1} \}
\ge \nu_k$. Then we have
\[
\beta_{k+1} = \frac{\min \{ \underline{x}^{k+1}, \underline{s}^{k+1} \}}{\nu_{k+1}} 
\ge \frac{ \nu_k}{\nu_k (1- \sin(\alpha_k))} \ge 1. 
\]
\end{itemize}
Adjoining these cases, we conclude $\beta \ge \min \{ \beta_0, 1 \}$.
\hfill
\qed
\end{proof}

The main purpose of the rest section is to establish a polynomial 
bound for this algorithm. In view of (\ref{neiborArc}) and 
(\ref{infeasible1}), to show the convergence of Algorithm 
\ref{mainAlgo}, we need to show that there is a sequence of
${\alpha}_k \in (0,\pi/2]$ with $\sin(\alpha_k)$ being
bounded by a polynomial of $n$
and $\sigma_k \in [\sigma_{\min},\sigma_{max}]$ such that (a) 
$r_b^k \rightarrow 0$, $r_c^k \rightarrow 0$ (which has been 
shown in Lemma \ref{basic}), and $\mu_k \rightarrow 0$, 
(b) $(x^k,s^k)>0$  for all $k \ge 0$, and
(c) $x(\sigma_k, \alpha_k) s(\sigma_k, \alpha_k) \ge 
\theta \mu(\sigma_k, \alpha_k)e$ for all $k \ge 0$ . 

Although our strategy is similar to the one used by Kojima 
\cite{kojima96}, Kojima, Megiddo, and Mizuno \cite{kmm93}, 
Wright \cite{wright97}, and Zhang \cite{zhang94}, our convergence 
result does not depend on some unrealistic and unnecessary 
restrictions assumed in those papers. We start with a simple 
but important observation. We will use the definition 
$D= X^{\frac{1}{2}} S^{-\frac{1}{2}}=\diag(D_{ii})$. 

\begin{lemma} 
For Algorithm \ref{mainAlgo}, there is a constant $C_1$ independent 
of $n$ such that for $\forall i \in \{ 1, \ldots, n \}$
\begin{equation}
(D^{k}_{ii})^{-1}  \nu_k = \nu_k \sqrt{\frac{s_i^k}{x_i^k}}
\le C_1 \sqrt{n \mu_k}, \hspace{0.1in}
D_{ii}^k  \nu_k = \nu_k \sqrt{\frac{x_i^k}{s_i^k}}
\le C_1 \sqrt{n \mu_k}.
\end{equation}
\label{cond1}
\end{lemma}
\begin{proof}
We know that $\min \{ \underline{x}^k, \underline{s}^k \}>0$, $\nu_k > 0$,
and $\beta>0$ is a constant independent of $n$. By the 
definition of $\beta_k$, we have 
$x_i^k \ge \underline{x}^k \ge \beta_k \nu_k \ge \beta \nu_k$ and 
$s_j^k \ge \underline{s}^k \ge \beta_k \nu_k \ge \beta \nu_k$.
This gives, for $\forall i \in \{1, \ldots, n\}$,
\begin{equation}
(D^{k}_{ii})^{-1}  \nu_k = \sqrt{\frac{s_i^k}{x_i^k}} \nu_k 
\le \frac{1}{\beta} \sqrt{s_i^k x_i^k} \le \frac{1}{\beta} 
\sqrt{n \mu_k} :=C_1 \sqrt{n \mu_k}. 
\nonumber 
\end{equation}
Using a similar argument for $s_j^k \ge \beta \nu_k$, we can show
that $ D_{ii} \nu_k \le C_1 \sqrt{n \mu_k}$.
\hfill \qed
\end{proof}

The main idea in the proof is based on a crucial observation used 
in many literatures, for example, Mizuno \cite{mizuno94} 
and Kojima \cite{kojima96}. We include the discussion here for completeness.
Let $(\bar{x}, \bar{\lambda}, \bar{s})$ be a feasible point satisfying 
$A\bar{x} =b$ and $A^{\Tr} \bar{\lambda}+ \bar{s}=c$. The 
existence of $(\bar{x}, \bar{\lambda}, \bar{s})$ is guaranteed by 
Assumption 2. We will make an additional assumption in the rest
discussion.
\newline
\newline
{\bf Assumption 5:}
\begin{itemize}
\item[] {\it There exist a big constant $M$ which is independent 
to the problem size $n$ and $m$ such that 
$\| \left( x^0-\bar{x}, s^0-\bar{s} \right) \| < M$. }
\end{itemize}

Since 
\[
A \dot{x} = r_b^k = \nu_k r_b^0 = \nu_k (A x^0 -b) = \nu_k A (x^0 - \bar{x}),
\]
we have
\[
A(\dot{x} - \nu_k (x^0-\bar{x}))=0.
\]
Similarly, since 
\[
A^{\Tr} \dot{\lambda} + \dot{s} = r_c^k = \nu_k r_c^0 = \nu_k (A^{\Tr}  \lambda^0 +s^0-c) 
= \nu_k ( A^{\Tr} (\lambda^0 - \bar{\lambda})+(s^0 - \bar{s})),
\]
we have
\[
A^{\Tr} (\dot{\lambda} - \nu_k (\lambda^0 - \bar{\lambda}))+(\dot{s} -\nu_k (s^0 - \bar{s}))=0.
\]
Using Lemma \ref{simple}, we have
\[
s \circ (\dot{x} - \nu_k (x^0 - \bar{x}))
+x \circ (\dot{s} -\nu_k (s^0 - \bar{s}))
= x \circ s - \nu_k s \circ (x^0 - \bar{x}) - \nu_k x \circ (s^0 - \bar{s}).
\]
Thus, in matrix form, we have
\begin{equation}
\left[ \begin{array}{ccc}
A & 0 & 0 \\
0 & A^{\Tr} & I \\
S & 0 & X
\end{array} \right] 
\left[ \begin{array}{ccc}
\dot{x} - \nu_k (x^0-\bar{x}) \\
\dot{\lambda} - \nu_k (\lambda^0 - \bar{\lambda}) \\
\dot{s} -\nu_k (s^0 - \bar{s})
\end{array} \right] 
=
\left[ \begin{array}{c}
0 \\  0  \\  x \circ s - \nu_k s \circ (x^0 - \bar{x}) - \nu_k x \circ (s^0 - \bar{s}) 
\end{array} \right] 
\label{matrixForm}
\end{equation}
Denote $(\delta x, \delta \lambda, \delta s) = (\dot{x} - \nu_k (x^0-\bar{x}),
\dot{\lambda} - \nu_k (\lambda^0 - \bar{\lambda}), \dot{s} -\nu_k (s^0 - \bar{s}))$
and 
$r= r^1 + r^2+ r^3$ with $(r^1, r^2, r^3)=(x \circ s, - \nu_k s \circ (x^0 - \bar{x}),
- \nu_k x \circ (s^0 - \bar{s}))$. For $i=1,2,3$, let $(\delta x^i, \delta \lambda^i, \delta s^i)$
be the solution of 
\begin{equation}
\left[ \begin{array}{ccc}
A & 0 & 0 \\
0 & A^{\Tr} & I \\
S & 0 & X
\end{array} \right] 
\left[ \begin{array}{ccc}
\delta x^i \\
\delta {\lambda}^i \\
\delta s^i
\end{array} \right] 
=
\left[ \begin{array}{c}
0 \\  0  \\  r^i 
\end{array} \right] 
\label{matrixSplit}
\end{equation}
Clearly, we have 
\begin{subequations}
\begin{gather}
\delta x = \delta x^1+\delta x^2+\delta x^3 =\dot{x} - \nu_k (x^0-\bar{x}), \label{deltaX} \\
\delta \lambda = \delta \lambda^1+\delta \lambda^2+\delta \lambda^3 
=\dot{\lambda} - \nu_k (\lambda^0-\bar{\lambda}) \label{deltaL}, \\
\delta s = \delta s^1+\delta s^2+\delta s^3 =\dot{s} - \nu_k (s^0-\bar{s}). \label{deltaS} 
\end{gather}
\end{subequations} 
From the second row of (\ref{matrixSplit}), we have 
$(D^{-1} \delta x^i)^{\Tr} (D \delta s^i) = 0$, 
for $i=1,2,3$, therefore,
\begin{equation}
\| D^{-1} \delta x^i \|^2, \| D \delta s^i \|^2 \le \| D^{-1} \delta x^i \|^2 + \| D \delta s^i \|^2
= \|D^{-1} \delta x^i + D \delta s^i \|^2.
\label{general}
\end{equation}
Applying $S \delta x^i + X \delta s^i= r^i$ to (\ref{general}) 
for $i=1,2,3$ respectively, we obtain the following relations
\begin{subequations}
\begin{gather}
\| D^{-1} \delta x^1 \|, \| D \delta s^1 \| \le  \| D^{-1} \delta x^1 + D \delta s^1 \|
=\| (Xs)^{\frac{1}{2}} \| = \sqrt{x^{\Tr}s} = \sqrt{n\mu}, \label{deltaX1S1} \\
\| D^{-1} \delta x^2 \|, \| D \delta s^2 \| \le  \| D^{-1} \delta x^2 + D \delta s^2 \|
=\nu_k \| D^{-1} (x^0-\bar{x}) \| ,  \label{deltaX2S2} \\
\| D^{-1} \delta x^3 \|, \| D \delta s^3 \| \le  \| D^{-1} \delta x^3 + D \delta s^3 \|
=\nu_k \| D (s^0-\bar{s}) \|. \label{deltaX3S3} 
\end{gather}
\label{deltaXS} 
\end{subequations} 
Considering (\ref{matrixSplit}) with $i=2$, we have
\[
S\delta x^2 + X\delta s^2 = r^2 = -\nu_k S(x^0 - \bar{x}),
\]
which is equivalent to 
\begin{equation}
\delta x^2 = -\nu_k (x^0 - \bar{x}) - D^2 \delta s^2.
\label{deltax2}
\end{equation}
Thus, from (\ref{deltaX}), (\ref{deltax2}), and (\ref{deltaXS}), we have
\begin{eqnarray}
\| D^{-1} \dot{x} \| & = & 
\| D^{-1} [\delta x^1 +\delta x^2 +\delta x^3 + \nu_k (x^0 - \bar{x})] \|
\nonumber \\
& = & \| D^{-1} \delta x^1 - D \delta s^2 + D^{-1} \delta x^3 \|
\nonumber \\
& \le & \| D^{-1} \delta x^1 \| + \| D \delta s^2  \| + \| D^{-1} \delta x^3 \|.
\end{eqnarray}
Considering (\ref{matrixSplit}) with $i=3$, we have
\[
S\delta x^3 + X\delta s^3 = r^3 = -\nu_k X(s^0 - \bar{s}),
\]
which is equivalent to 
\begin{equation}
\delta s^3 = -\nu_k (s^0 - \bar{s}) - D^{-2} \delta x^3.
\label{deltas3}
\end{equation}
Thus, from (\ref{deltaS}), (\ref{deltas3}), and (\ref{deltaXS}), we have
\begin{eqnarray}
\| D  \dot{s} \| & = & 
\| D  [\delta s^1 +\delta s^2 +\delta s^3 + \nu_k (s^0 - \bar{s})] \|
\nonumber \\
& = & \| D  \delta s^1 + D \delta s^2 - D^{-1} \delta x^3 \|
\nonumber \\
& \le & \| D  \delta s^1 \| + \| D \delta s^2  \| + \| D^{-1} \delta x^3 \|.
\end{eqnarray}

From (\ref{deltaX1S1}), we can summarize the above discussion as the following (cf. \cite{kojima96})
\begin{lemma}
Let $(x^0, \lambda^0, s^0)$ be the initial point of Algorithm \ref{mainAlgo},
$\bar{x}$ be a feasible solution of (\ref{LP}), and $(\bar{\lambda}, \bar{s})$ be a 
feasible solution of (\ref{DP}). Then
\begin{equation}
\| D  \dot{s} \|, \| D^{-1}  \dot{x} \|
\le \sqrt{n\mu} + \| D \delta s^2  \| + \| D^{-1} \delta x^3 \|.
\label{mainIneq}
\end{equation}
\label{wright}
\end{lemma}

\begin{remark}
If the initial point $(x^0, \lambda^0, s^0)$ is a feasible point 
satisfying $A x^0=b$ and $A^{\Tr} \lambda^0 + s^0 = c$, then the 
problem is reduced to a feasible interior-point problem which has 
been discussed in \cite{yang13}. In this case, inequality (\ref{mainIneq}) is 
reduced to $\| D  \dot{s} \|, \| D^{-1}  \dot{x} \| \le \sqrt{n\mu}$
because $x^0=\bar{x}$, $s^0=\bar{s}$, and 
$\| D \delta s^2  \| =\| D^{-1} \delta x^3 \| =0$ from
(\ref{deltaX2S2}) and (\ref{deltaX3S3}). 
Using $\| D  \dot{s} \|, \| D^{-1}  \dot{x} \| \le \sqrt{n\mu}$, 
we have proved \cite{yang13} that a feasible arc-search algorithm 
is polynomial with complexity bound ${\mathcal O}(\sqrt{n}\log(1/\epsilon))$. 
In the remainder of the paper, we will focus on the case that
the initial point is infeasible.
\end{remark}

%Let $I \in \{ 1, \ldots, n \}$ and $J \in \{ 1, \ldots, n \}$
%such that $x_I = \max_i x_i$, $s_J = \max_j y_j$.

\begin{lemma}
Let $(\dot{x}, \dot{\lambda}, \dot{s})$ be 
defined in (\ref{doty}). Then, there is a constant $C_2$ independent of $n$ 
such that in every iteration of Algorithm \ref{mainAlgo}, the follwoing inequality holds.
\begin{equation}
\| D  \dot{s} \|, \| D^{-1}  \dot{x} \|
\le C_2 \sqrt{n \mu}.
\label{converg2}
\end{equation} 
\label{main1}
\end{lemma}
\begin{proof}
Since $D$ and $D^{-1}$ are diagonal matrices, in view of (\ref{matrixSplit}), we have 
\[ (D\delta s^2)^{\T} (D^{-1} \delta x^3)=(\delta s^2)^{\T} (\delta x^3)=0. \]
Let $(x^0, \lambda^0, s^0)$ be the initial point of Algorithm 
\ref{mainAlgo}, $\bar{x}$ be a feasible solution of (\ref{LP}), 
and $(\bar{\lambda}, \bar{s})$ be a feasible solution of 
(\ref{DP}). Then, from (\ref{deltaX2S2}) and 
(\ref{deltaX3S3}), we have
\begin{eqnarray}
& & \| D\delta s^2- D^{-1} \delta x^3 \|^2 \nonumber \\
& = & 
\| D\delta s^2 \|^2 + \| D^{-1} \delta x^3 \|^2
\nonumber \\
& = & \nu_k^2 \| D^{-1} (x^0-\bar{x})  \|^2 + \nu_k^2 \| D (s^0-\bar{s}) \|^2 
\nonumber \\
& = & \nu_k^2 \left[(x^0-\bar{x})^{\T} \diag \left( \frac{s_i^k}{x_i^k} \right) (x^0-\bar{x})
+(s^0-\bar{s})^{\T} \diag \left( \frac{x_i^k}{s_i^k} \right) (s^0-\bar{s}) \right]
\nonumber \\
& = & \nu_k^2 \left( x^0-\bar{x}, s^0-\bar{s} \right)^{\T} 
\left[ \begin{array}{cc} 
\diag \left( \frac{s_i^k}{x_i^k} \right) & 0 \\
0 & \diag \left( \frac{x_i^k}{s_i^k} \right) \end{array} \right]
\left( x^0-\bar{x}, s^0-\bar{s} \right)
\nonumber \\
& \le & \nu_k^2 \max_i \left\{ \frac{s_i^k}{x_i^k}, \frac{x_i^k}{s_i^k} \right\} 
\| \left( x^0-\bar{x}, s^0-\bar{s} \right) \|^2 \nonumber \\
& = & \nu_k^2 \max_i \{ D_{ii}^{-2}, D_{ii}^{2} \} 
\| \left( x^0-\bar{x}, s^0-\bar{s} \right) \|^2 \nonumber \\
& \le & C_1^2 n \mu_k \| \left( x^0-\bar{x}, s^0-\bar{s} \right) \|^2,
\end{eqnarray}
where the last inequality follows from Lemma \ref{cond1}.
Adjoining this result with Lemma \ref{wright} and Assumption 5 gives
\[
\| D  \dot{s} \|, \| D^{-1}  \dot{x} \|
\le \sqrt{n\mu_k} + \| \left( x^0-\bar{x}, s^0-\bar{s} \right) \| C_1\sqrt{n \mu_k} 
\le C_2 \sqrt{n \mu_k}.
\]
This finishes the proof.
\hfill
\qed
\end{proof}

From Lemma \ref{main1}, we can obtain several inequalities that will be 
used in our convergence analysis. The first one is given as follows.

\begin{lemma}
Let $(\dot{x}, \dot{\lambda}, \dot{s})$ and $(\ddot{x}, \ddot{\lambda}, \ddot{s})$ be 
defined in (\ref{doty}) and (\ref{ddoty}). 
Then, there exists a constant
$C_3>0$ independent of $n$ such that the following relations hold.
\begin{subequations}
\begin{gather}
\| D^{-1} \ddot{x} \|, \| D  \ddot{s} \| \le C_3 n  \mu_k^{0.5}, 
\\
\| D^{-1} p_{x} \|, \| D  p_{s} \| \le \sqrt{\frac{n}{\theta}} \mu_k^{0.5}, 
\\
\| D^{-1} q_{x} \|, \| D  q_{s} \| \le 
\frac{2C_2^2 }{\sqrt{\theta}} n \mu_k^{0.5}.
\end{gather}
\end{subequations}
\label{extension}
\end{lemma}
\begin{proof}
From the last row of (\ref{barq}), using the facts that $q_{x}^{\T} q_{s} = 0$, 
$x_i^k s_i^k > \theta \mu_k$, and Lemma
\ref{main1}, we have
\begin{eqnarray}
& & S q_{x} + X q_{s} = -2 \dot{x} \circ \dot{s}  \nonumber \\
\Longleftrightarrow & & D^{-1} q_{x} + D q_{s} = 2 (XS)^{-0.5} (-\dot{x} \circ \dot{s} )
=2 (XS)^{-0.5} (-D^{-1} \dot{x} \circ D \dot{s} )
 \nonumber \\
\Longrightarrow & & \| D^{-1} q_{x}\|^2,\| D q_{s} \|^2 
\le \| D^{-1} q_{x}\|^2   + \| D q_{s} \|^2 
= \| D^{-1}  q_{x} + D q_{s} \|^2   \nonumber \\
& \le & 4 \| (XS)^{-0.5} \|^2 \left( \| D^{-1}  \dot{x} \| \cdot \| D \dot{s} \| \right)^2
 \nonumber \\
& \le & \frac{4}{\theta \mu_k} \left( C_2^2 n \mu_k \right)^2
=\frac{(2C_2^2n)^2}{\theta}\mu_k. \nonumber 
\end{eqnarray}
Taking the square root on both sides gives
\begin{equation}
\| D^{-1} q_{x}\|,\| D q_{s} \| \le \frac{2 C_2^2}{\sqrt{\theta}} n \sqrt{\mu_k}.
\label{ddotnorm1}
\end{equation}
From the last row of (\ref{barp}), using the facts that $p_{x}^{\T} p_{s} = 0$ and 
$x_i^k s_i^k \ge \theta \mu_k$, we have
\begin{eqnarray}
& & S p_{x} + X p_{s} = \mu_k e  \nonumber \\
\Longleftrightarrow & & D^{-1} p_{x} + D p_{s} = (XS)^{-0.5} \mu_k e \nonumber \\
\Longrightarrow & & \| D^{-1} p_{x}\|^2,\| D p_{s} \|^2 \le 
\| D^{-1} p_{x}\|^2   + \| D p_{s} \|^2 
\nonumber \\
& & =\| D^{-1}  p_{x} + D p_{s} \|^2     
\le \| (XS)^{-0.5} \|^2 n (\mu_k)^2 \nonumber \\
& &
\le n \mu_k/\theta. \nonumber 
\end{eqnarray}
Taking the square root on both sides gives
\begin{equation}
\| D^{-1} p_{x}\|,\| D p_{s} \| \le \sqrt{\frac{n}{\theta}} \sqrt{\mu_k}.
\label{ddotnorm2}
\end{equation}
Combining (\ref{ddotnorm1}) and (\ref{ddotnorm2}) proves the lemma.
\hfill \qed
\end{proof}

The following inequalities are direct results of Lemmas \ref{main1} and \ref{extension}.

\begin{lemma}
Let $(\dot{x}, \dot{\lambda}, \dot{s})$ and $(\ddot{x}, \ddot{\lambda}, \ddot{s})$ be defined in (\ref{doty}) and (\ref{ddoty}). 
Then, the following relations hold.
\begin{equation}
\frac{| \dot{x}^{\T} \dot{s} |}{n} \le C_2^2 \mu_k, \hspace{0.1in}
\frac{| \ddot{x}^{\T} \dot{s}|}{n} \le C_2C_3 \sqrt{n} \mu_k, \hspace{0.1in}
\frac{| \dot{x}^{\T} \ddot{s}|}{n} \le C_2C_3 \sqrt{n} \mu_k. 
\label{rela1}
\end{equation}
Moreover,
\begin{equation}
| \dot{x}_i \dot{s}_i | \le C_2^2 n \mu_k, \hspace{0.1in}
| \ddot{x}_i \dot{s}_i| \le C_2C_3 n^{\frac{3}{2}} \mu_k, \hspace{0.1in}
| \dot{x}_i \ddot{s}_i| \le C_2C_3 n^{\frac{3}{2}} \mu_k, \hspace{0.1in}
| \ddot{x}_i \ddot{s}_i| \le C_3^2 n^2 \mu_k. 
\label{rela2}
\end{equation}
\label{inequalities}
\label{lemma5}
\end{lemma}
\begin{proof}
The first relation of (\ref{rela1}) is given as follows.
\begin{eqnarray}
\frac{| \dot{x}^{\T} \dot{s} |}{n} = \frac{| (D^{-1} \dot{x})^{\T} (D \dot{s}) |}{n}
\le \frac{ \| D^{-1}  \dot{x} \| \cdot  \| D \dot{s}\|}{n}  \le  C_2^2 \mu_k.
\end{eqnarray}
Similarly, we have
\begin{eqnarray}
\frac{| \ddot{x}^{\T} \dot{s} |}{n} = \frac{| (D^{-1} \ddot{x})^{\T} (D \dot{s}) |}{n}
\le \frac{ \| D^{-1}  \ddot{x} \| \cdot  \| D \dot{s}\|}{n}  \le  C_2C_3 \sqrt{n}\mu_k,
\end{eqnarray}
and
\begin{eqnarray}
\frac{| \dot{x}^{\T} \ddot{s} |}{n} = \frac{| (D^{-1} \dot{x})^{\T} (D \ddot{s}) |}{n}
\le \frac{ \| D^{-1}  \dot{x} \| \cdot  \| D \ddot{s}\|}{n}  \le  C_2C_3 \sqrt{n}\mu_k.
\end{eqnarray}
The first relation of (\ref{rela2}) is given as follows.
\begin{eqnarray}
| \dot{x}_i \dot{s}_i | = | D_{ii}^{-1} \dot{x}_i D_{ii} \dot{s}_i |
\le | D_{ii}^{-1} \dot{x}_i | \cdot | D_{ii} \dot{s}_i |
\le \| D^{-1}  \dot{x} \| \cdot  \| D \dot{s}\| \le  C_2^2 n \mu_k.
\end{eqnarray}
Similar arguments can be used for the rest inequalities of (\ref{rela2}).
\hfill \qed
\end{proof}

Now we are ready to show that there exists a constant 
$\kappa_0={\mathcal O}({n}^{\frac{3}{2}})$ such that for every 
iteration, for some $\sigma_k \in [\sigma_{\min},\sigma_{\max}]$ and
all $\sin(\alpha_k) \in (0,\frac{1}{\kappa_0}]$, all conditions 
in Step 2 of Algorithm \ref{mainAlgo} hold.

\begin{lemma} There exists a positive constant $C_4$ independent of $n$ and
an $\bar{\alpha}$ defined by $\sin(\bar{\alpha}) \ge \frac{C_4}{\sqrt{n}}$ such that
for $\forall k \ge 0$ and $\sin(\alpha_k) \in (0,\sin(\bar{\alpha})]$,
\begin{equation}
(x_i^{k+1},s_i^{k+1}):=(x_i(\sigma_k,\alpha_k), s_i(\sigma_k,\alpha_k) ) 
\ge (\phi_k, \psi_k) >0
\label{posity}
\end{equation}
holds.
\label{positiveCond}
\end{lemma}
\begin{proof}
From (\ref{posi1}) and (\ref{phi}), a conservative estimation can be obtained by 
\begin{eqnarray}
x_i(\sigma_k,\alpha_k) = x_i^k-\dot{x}_i \sin(\alpha_k) +\ddot{x}_i(1-\cos(\alpha_k)) \ge \rho x_i^k
\nonumber
\end{eqnarray}
which is equivalent to 
\[
x_i^k(1-\rho)-\dot{x}_i \sin(\alpha_k) +\ddot{x}_i(1-\cos(\alpha_k)) \ge 0
\]
Multiplying $D_{ii}^{-1}$ to this inequality and using Lemmas \ref{main1}, 
\ref{inequalities}, and \ref{sincos}, we have 
\begin{eqnarray}
(x_i^ks_i^k)^{0.5}(1-\rho)-D_{ii}^{-1} \dot{x}_i \sin(\alpha_k) +D_{ii}^{-1} \ddot{x}_i(1-\cos(\alpha_k)) 
\nonumber \\
\ge  \sqrt{\mu_k } \left(\sqrt{\theta} (1-\rho)-C_2\sqrt{n}\sin(\alpha_k) -C_3n \sin^2(\alpha_k) \right)
\nonumber
\end{eqnarray}
Clearly, the last expression is greater than zero for all 
$ \sin(\alpha_k) \le \frac{C_4}{\sqrt{n}} \le \sin(\bar{\alpha})$, 
where $C_4 = \frac{\sqrt{\theta} (1-\rho)}{2\max\{C_2, \sqrt{C_3}\}}$. 
This proves $x_i^{k+1} \ge \phi_k >0$. Similarly, we have $s_i^{k+1} \ge \psi_k>0$.
\hfill \qed
\end{proof}

\begin{lemma} 
There exists a positive constant $C_5$ independent of $n$ and
an $\hat{\alpha}$ defined by $\sin(\hat{\alpha}) \ge \frac{C_5}{n^{\frac{1}{4}}}$ such that
for $\forall k \ge 0$ and $\sin(\alpha) \in (0,\sin(\hat{\alpha})]$, the following relation
\begin{equation}
\mu_k (\sigma_k,\alpha_k) \le \mu_k \left( 1 - \frac{\sin(\alpha_k)}{4} \right)
\le \mu_k \left( 1 - \frac{C_5}{4n^{\frac{1}{4}}} \right)
\label{polynoty}
\end{equation}
\label{objDec}
\end{lemma}
holds.
\begin{proof}
Using (\ref{diffstepsize}), Lemmas \ref{sincos} and \ref{inequalities}, we have
\begin{eqnarray}
\mu (\sigma_k, \alpha_k) & = & \mu_k (1- \sin(\alpha_k)) +\sigma_k \mu_k (1-\cos(\alpha_k)) 
-  \frac{ \dot{x}^{\T}\dot{s} }{n} (1-\cos(\alpha_k))^2
\nonumber \\
& & - \frac{ \dot{x}^{\T} \ddot{s}  
    + \dot{s}^{\T} \ddot{x} }{n}\sin(\alpha_k) (1-\cos(\alpha_k))
\nonumber \\
& \le & \mu_k \left[
1-\sin(\alpha_k) + \sigma_k \sin^2(\alpha_k) \right] 
\nonumber \\
& & + \left[ \frac{\lvert \dot{x}^{\T} \dot{s}\lvert }{n} \sin^4(\alpha_k)
+ \left(\frac{\lvert \dot{x}^{\T} \ddot{s}\lvert }{n} 
+ \frac{\lvert \ddot{x}^{\T} \dot{s}\lvert }{n} \right) \sin^3(\alpha_k)
\right]
\nonumber \\
& \le & \mu_k \left[
1-\sin(\alpha_k) + \sigma_k \sin^2(\alpha_k) + C_2^2 \sin^4(\alpha_k)+2C_2C_3 \sqrt{n} \sin^3(\alpha_k) \right]
\nonumber \\
& = & \mu_k \left[
1-\sin(\alpha_k) \left( 1 - \sigma_k \sin(\alpha_k) - C_2^2 \sin^3(\alpha_k)-2C_2C_3 \sqrt{n} \sin^2(\alpha_k)
\right) \right].
\nonumber 
\end{eqnarray}
Let 
\[
C_5=\frac{1}{5 \max \{ \sigma_k, C_2^2, 2C_2C_3 \}}.
\]
then, for all 
$ \sin(\alpha_k) \in (0, \sin(\hat{\alpha})]$ 
and $\sigma_{\min} \le  \sigma_k \le \sigma_{\max}$, 
inequality (\ref{polynoty}) holds.
\hfill\qed
\end{proof}

\begin{lemma} 
There exists a positive constant $C_6$ independent of $n$ and an
$\check{\alpha}$ defined by $\sin(\check{\alpha}) \ge \frac{C_6}{n^{\frac{3}{2}}}$ 
such that if $x_i^{k}s_i^{k} \ge \theta \mu_{k}$ holds, then for $\forall k \ge 0$,
$\forall i \in \{ 1, \ldots, n\}$, and $\sin(\alpha) \in (0,\sin(\check{\alpha})]$, the following relation
\begin{equation}
x_i^{k+1}s_i^{k+1} \ge \theta \mu_{k+1}
\label{polycompl}
\end{equation}
\label{barrier}
\end{lemma}
holds.
\begin{proof}
Using (\ref{thirdBound1}) and (\ref{diffstepsize}), we have
\begin{eqnarray}
& & x_i^{k+1} s_i^{k+1} - \theta \mu_{k+1} \nonumber \\
& \ge & \sigma_k \mu_k (1-\theta)  (1-\cos(\alpha_k))  \nonumber \\
&  & 
-\left[ \ddot{x}_i^k \dot{s}_i^k+\dot{x}_i^k \ddot{s}_i^k - 
\frac{ \theta ( \dot{x}^{\T} \ddot{s}  + \dot{s}^{\T} \ddot{x}) }{n} \right]
 \sin(\alpha_k) (1-\cos(\alpha_k))
 \nonumber \\ &  & 
+ \left[ \ddot{x}_i^k \ddot{s}_i^k-\dot{x}_i^k \dot{s}_i^k 
+\frac{ \theta (\dot{x}^{\T}\dot{s} ) }{n} \right]
(1-\cos(\alpha_k))^2
\end{eqnarray}
Therefore, if 
\begin{eqnarray}
\sigma_k \mu_k (1-\theta) & - &
\left[ \ddot{x}_i^k \dot{s}_i^k+\dot{x}_i^k \ddot{s}_i^k - 
\frac{ \theta ( \dot{x}^{\T} \ddot{s}  + \dot{s}^{\T} \ddot{x}) }{n} \right]
 \sin(\alpha_k)  \nonumber \\
& + & \left[ \ddot{x}_i^k \ddot{s}_i^k-\dot{x}_i^k \dot{s}_i^k 
+\frac{ \theta (\dot{x}^{\T}\dot{s} ) }{n} \right]
(1-\cos(\alpha_k)) \ge 0,
\label{condition}
\end{eqnarray}
then
\[
 x_i^{k+1} s_i^{k+1} \ge \theta \mu_{k+1}.
\]
The inequality (\ref{condition}) holds if
\begin{eqnarray}
\sigma_k \mu_k (1-\theta) 
& - & \left( |\ddot{x}_i^k \dot{s}_i^k|+|\dot{x}_i^k \ddot{s}_i^k|
+\Big| \frac{ \theta ( \dot{x}^{\T} \ddot{s}  + \dot{s}^{\T} \ddot{x}) }{n} \Big| \right) \sin(\alpha_k)
\nonumber \\
& - & \left( |\ddot{x}_i^k \ddot{s}_i^k|+|\dot{x}_i^k \dot{s}_i^k |
+\Big| \frac{ \theta (\dot{x}^{\T}\dot{s} ) }{n} \Big| \right)
 \sin^2(\alpha_k) \ge 0. \nonumber
\end{eqnarray}
Using Lemma \ref{lemma5}, we can easily find some $\check{\alpha}$ 
defined by $\sin(\check{\alpha}) \ge \frac{C_6}{n^{\frac{3}{2}}}$ 
to meet the above inequality.
\hfill\qed
\end{proof}

Now the convergence result follows from the standard argument
using the following theorem given in \cite{wright97}.

\begin{theorem}
Let $\epsilon \in (0,1)$ be given. Suppose that an algorithm generates a sequence of iterations 
$\{ \chi_k \}$ that satisfies
\begin{equation}
\chi_{k+1} \le \left( 1 - \frac{\delta}{n^{\omega}} \right) \chi_k, 
\hspace{0.1in} k=0, 1, 2, \ldots,
\end{equation}
for some positive constants $\delta$ and $\omega$. Then there 
exists an index $K$ with
\[
K={\mathcal O}(n^{\omega}\log({\chi_0}/{\epsilon}))
\]
such that 
\[
\chi_k \le \epsilon \hspace{0.1in} {\rm for} \hspace{0.1in} \forall k \ge K.
\]
\hfill \qed
\label{white97}
\end{theorem}

In view of Lemmas \ref{basic}, \ref{positiveCond}, \ref{objDec}, 
\ref{barrier}, and Theorem \ref{white97}, 
we can state our main result as the following 
\begin{theorem}
Algorithms~\ref{mainAlgo} is a polynomial algorithm with
polynomial complexity bound of 
${\mathcal O}({n}^{\frac{3}{2}} \max \{ \log({(x^0)^{\T}s^0}/{\epsilon}), \log({r_b^0}/{\epsilon}),\log({r_c^0}/{\epsilon}) \} )$.
\end{theorem}

%\begin{remark}
%Because Algorithm \ref{mainAlgo} guarantees that $x_i^ks_i^k \ge 
%\mu_k$, using the same argument of \cite{wright97}, we can show
%that Algorithm \ref{mainAlgo} always finds a strict complementary 
%solution for the linear programming problems. 
%This is important
%for numerical computation because interior-point algorithms have
%difficult near degenerate solutions \cite{yang16}. We will 
%discuss this further in Section \ref{implSec}.7.
%\end{remark}

\section{Algorithm 2}

This algorithm is a simplified version of Algorithm \ref{mainAlgo}. 
The only difference of the two algorithm is that this algorithm
searches optimizers in a larger neighborhood defined by
(\ref{infeasible2}). From the discussion in Section~\ref{arcSearch},
the following arc-search infeasible interior-point algorithm is well-defined.

\begin{algorithm} {\ } \\ 
Data: $A$, $b$, $c$.  \hspace{0.1in} {\ } \\
Parameter: $\epsilon \in (0,1)$, $\sigma_{\min} \in (0,1)$, 
$\sigma_{\max} \in (0,1)$, and $\rho \in (0,1)$. {\ } \\
Initial point: ${\lambda}^0=0$ and $(x^0, s^0)>0$.
%, and $0< \nu_0 < \displaystyle\min_i \{ x_i^0, s_i^0 \}$.
\newline
{\bf for} iteration $k=0,1,2,\ldots$
\begin{itemize}
\item[] Step 0: If ${\| r_b^k \|} \le {\epsilon}$, ${\| r_c^k \|} \le {\epsilon}$, 
and $\mu_k \le \epsilon$, stop.
\item[] Step 1: Calculate ${\mu}_k$, $r_b^k$, $r_c^k$, $\dot{\lambda}$, $\dot{s}$, 
$\dot{x}$, $p_x^k$, $p_{\lambda}^k$, $p_s^k$, $q_x^k$, $q_{\lambda}^k$, and $q_s^k$. 
\item[] Step 2: Find some appropriate $\alpha_k \in (0,\pi/2]$
and $\sigma_k \in [\sigma_{\min},\sigma_{\max}]$ to satisfy 
\begin{equation}
x(\sigma_k, \alpha_k) \ge \phi_k e, \hspace{0.1in} 
s(\sigma_k, \alpha_k) \ge \psi_k e, \hspace{0.1in}  
\mu_k > \mu(\sigma_k, \alpha_k).  
 \label{stepCond}
\end{equation}
\item[] Step 3: Set $(x^{k+1}, \lambda^{k+1}, s^{k+1})
=(x(\sigma_k, \alpha_k), \lambda(\sigma_k, \alpha_k), s(\sigma_k, \alpha_k))$ and 
$\mu_{k+1}=\mu(\sigma_k, \alpha_k)$.
\item[] Step 4: Set $k+1 \rightarrow k$. Go back to Step 1.
\end{itemize}
{\bf end (for)} 
\hfill \qed
\label{mainAlgo2}
\end{algorithm}

\begin{remark}
It is clear that the only difference between Algorithm~\ref{mainAlgo} and
Algorithm ~\ref{mainAlgo2} is in Step 2, where Algorithm~\ref{mainAlgo2} 
does not require $x(\sigma_k, \alpha_k)s(\sigma_k, \alpha_k) \ge \theta \mu(\sigma_k, \alpha_k)$. We have seen from Lemma \ref{barrier} that this
requirement is the main barrier to achieve a better polynomial bound.
\end{remark}

Denote
\begin{equation}
\gamma = \min \{ 1, \rho\beta \}.
\label{gamma}
\end{equation}

\begin{lemma}
If $ \{ \frac{\nu_k^2}{\mu_k} \} >0$ is bounded below from zero and
$\beta >0$, then, there is a positive constant $\theta >0$, such that 
$x^k_i s^k_i \ge \theta \mu_k$ for $\forall i \in \{ 1, 2, \ldots, n \}$ and $\forall k \ge 0$.
\end{lemma}
\begin{proof}
By the definition of $\beta$, we have $\beta \le \frac{\underline{x}^k }{\nu_k} \le \frac{x_i^k}{\nu_k}$ and $\beta \le \frac{\underline{s}^k }{\nu_k} \le \frac{s_i^k}{\nu_k}$, which can be written as
\[ \underline{x}^k \ge \beta \nu_k >0, \hspace{0.1in} 
\hspace{0.1in} \underline{s}^k \ge \beta \nu_k >0. \] 
Since $\phi_k = \min \{ \rho \underline{x}^k, \nu_k \}$, we have either
$\phi_k = \rho \underline{x}^k \ge \rho \beta \nu_k $ or
$\phi_k = \nu_k$, which means that 
\begin{equation}
\phi_k \ge \min \{ 1, \rho \beta \} \nu_k = \gamma \nu_k.
\end{equation}
Since $\psi_k = \min \{ \rho s_\jmath^k, \nu_k \}$, we can show
\begin{equation}
\psi_k \ge \min \{ 1, \rho \beta \} \nu_k = \gamma \nu_k.
\end{equation}
Using (\ref{posi1}) and (\ref{posi2}), the definition of $\phi_k$ and $\psi_k$, and the above two formulas, we have 
\[ x^{k}_i s^{k}_i \ge \phi_{k-1} \psi_{k-1} \ge \gamma^2 \nu_{k-1}^2 
> \gamma^2 \nu_{k-1}^2 (1-\sin(\alpha_{k-1}))^2 = \gamma^2 \nu_k^2 >0.
\]
Let $\theta = \displaystyle\inf_{k } \{ \frac{\gamma^2 \nu_k^2}{\mu_k} \}$, then, 
$\gamma^2 \{ \frac{\nu_k^2}{\mu_k} \} \ge \gamma^2 \displaystyle\inf_{k } \{ \frac{\nu_k^2}{\mu_k} \}= \theta >0$, and we have
\[
x^k_i s^k_i \ge \frac{\gamma^2 \nu_k^2}{\mu_k} \mu_k \ge \theta \mu_k.
\]
This finishes the proof.
\hfill
\qed
\end{proof}

Since $\nu_k>0$ for all iteration $k$ (see Assumption 4 and 
Section \ref{implSec}.10), we immediately have the following Corollary.
\begin{corollary}
If Algorithm \ref{mainAlgo2} terminates in finite iterations $K$, then
we have 
$\theta = \displaystyle\min_{k \le K} \{ \frac{\gamma^2 \nu_k }{\mu_k} \} >0$ 
is a constant independent of $n$, and
$x^k_i s^k_i \ge \theta \mu_k$ for $\forall i \in \{ 1, 2, \ldots, n \}$ and for $0 \le k \le K$.
\end{corollary}

\begin{proposition} 
Assume that $\|r_b^0\|$, $\|r_c^0\|$, and $\mu_0$ are all finite. 
Then, Algorithm \ref{mainAlgo2} terminates in finite iterations. 
\label{lastProp}
\end{proposition} 
\begin{proof}
Since $\|r_b^0\|$, $\|r_c^0\|$, and $\mu_0$ are finite, in view 
of Proposition 4.4, in every iteration, these variables decrease
at least a constant. Therefore, the algorithm will terminate in 
finite steps.
\hfill \qed
\end{proof}

Proposition \ref{lastProp}  implies that $x^k_i s^k_i \ge \theta \mu_k$ for
$\forall i \in \{ 1, 2, \ldots, n \}$ and for $0 \le k \le K$. 
$x^k_i s^k_i \ge \theta \mu_k$ is the most strict condition 
required in Section 4 to show the polynormiality of Algorithm
\ref{mainAlgo}. Since Algorithm \ref{mainAlgo2} does not check
the condition $x^k_i s^k_i \ge \theta \mu_k$, and it checks only
(\ref{stepCond}), from Lemmas \ref{positiveCond} and \ref{objDec}, 
we conclude

\begin{theorem}
If Algorithms~\ref{mainAlgo2} terminates in finite iterations,
it converges in polynomial iterations with a complexity bound of 
${\mathcal O}({n}^{\frac{1}{2}} \max \{ \log({(x^0)^{\T}s^0}/{\epsilon}), \log({r_b^0}/{\epsilon}),\log({r_c^0}/{\epsilon}) \} )$.
\label{mainTheorem2}
\end{theorem}

\section{Implementation details}
\label{implSec}

The two proposed algorithms are very similar to the one 
in \cite{yang16} in that they are all for solving linear 
programming of standard form using arc-search. 
Implementation strategies for the four algorithms (two proposed in 
this paper, one proposed in \cite{yang16}, and Mehrotra's
algorithm described in \cite{wright97}) are in common
except algorithm-specific parameters (Section 6.1), selection of 
$\alpha_k$ (Section 6.8 for arc-search), section of $\sigma_k$
(Section 6.9), and rescale $\alpha_k$ (Section 6.10) for the two 
algorithms proposed in this paper. 
Most of these implementation strategies 
have been thoroughly discussed in \cite{yang16}. Since all these 
strategies affect the numerical efficiency, we summarize all details 
implemented in this paper and explain the reasons why some strategies 
are adopted and some are not. 

\subsection{Default parameters}

Several parameters are used in Algorithms \ref{mainAlgo} and
\ref{mainAlgo2}. In our implementation, 
the following defaults are used without a serious effort to 
optimize the results for the test problems: 
$\theta=\min \{ 10^{-6},0.1*\min \{ x^0 \circ s^0  \}/\mu_0 \}$,
$\sigma_{\min} = 10^{-6}$, $\sigma_{\max} =0.4$ for Algorithm
\ref{mainAlgo}, $\sigma_{\max} =0.3$ for Algorithm 
\ref{mainAlgo2}, $\rho=0.01$, and $\epsilon=10^{-8}$. Note that 
$\theta$ is used only in Algorithm \ref{mainAlgo} and this 
$\theta$ selection guarantees $x^0_i s^0_i \ge \theta \mu_0$.  

\subsection{Initial point selection}

Initial point selection has been known an important factor in the computational efficiency for most infeasible interior-point 
algorithms \cite{cmww97, zhang96}. We use the methods 
proposed in \cite{Mehrotra92,lms92} to generate candidate
initial points. We then compare 
\begin{equation}
\max \{ \| A x^0-b\|, \| A^{\Tr} \lambda^0 +s^0 -c \|, \mu^0 \}
\label{initpoint}
\end{equation}
obtained by these two methods and select the initial point with 
smaller value of (\ref{initpoint}) as we guess this selection may 
reduce the number of iterations (see \cite{yang16} for detail).

\subsection{Pre-process and Post-process}

Pre-solver strategies for the standard linear programming 
problems represented in the form of (\ref{LP}) and solved in 
normal equations were fully investigated in \cite{yang16}. Five 
of them were determined to be effective and efficient in application. The same set of the pre-solver is used in this paper. The post-process is also the same as in \cite{yang16}.

\subsection{Matrix scaling}

Based on the test and analysis of \cite{yang16}, it is determined
that matrix scaling does not improve efficiency in general. Therefore, 
we will not use scaling in the implementation. But the ratio 
\begin{equation}
\frac{\max{|A_{i,j}|}}{\min { \{ | A_{k,l} |  A_{k,l} \ne 0 \} } }
\label{ratio}
\end{equation}
is used to determine if pre-process rule $9$ of \cite{yang16} 
is used.

\subsection{Removing row dependency from $A$}

Removing row dependency from $A$ is studied in \cite{andersen95}, 
Andersen reported an efficient method that removes
row dependency of $A$. Based on the study in \cite{yang16},
we choose to not use this function unless we feel it is necessary 
when it is used as part of handling degenerate solutions discussed below.
To have a fair comparison of all tested algorithms, we will 
make it clear in the test report what algorithms 
and/or problems use this function and what algorithms and/or 
problems do not use this function.

\subsection{Linear algebra for sparse Cholesky matrix}

Similar to Mehrotra's algorithm, the majority of the computational 
cost of the proposed algorithms is to solve sparse Cholesky 
systems (\ref{doy3}) and (\ref{detail}), which
can be expressed as an abstract problem as follows. 
\begin{equation}
AD^2A^{\T} u = L \Lambda L^{\Tr} u = v,
\label{useLater}
\end{equation}
where $D=X^{\frac{1}{2}}S^{-\frac{1}{2}}$ is identical in (\ref{doy3}) and (\ref{detail}), but $u$ and $v$ are different vectors. 
Many popular LP solvers \cite{cmww97,zhang96} call a 
software package \cite{ np93} which uses some linear algebra 
specifically developed for the sparse Cholesky decomposition 
\cite{liu85}. However, Matlab does not yet implemented
the function with the features for ill-conditioned matrices. We implement the same method as in \cite{yang16}.

\subsection{Handling degenerate solutions}

Difficulty caused by degenerate solutions in interior-point 
methods for linear programming has been realized for a long time
\cite{ghrt93}. Similar observation was also reported in 
\cite{gmstw86}. In our implementation, we have an option for 
handling degenerate solutions as described in \cite{yang16}.
%Since Algorithms \ref{mainAlgo} and \ref{mainAlgo2} find
%strict complementary solutions, we expect that this option will
%not be needed and we will discuss this further in Section 7.

\subsection{Analytic solution of $\alpha_k$}

Given $\sigma_k$, we know that $\alpha_k$ can be calculated in 
analytic form \cite{yang09,yang16}. Since Algorithms
\ref{mainAlgo} and \ref{mainAlgo2} are slightly different from 
the ones in \cite{yang09,yang16}, the formulas to calculate 
$\alpha_k$ are slightly different too. We provide these formulas 
without giving proofs. For each 
$i \in \lbrace 1,\ldots, n \rbrace$, we can select the largest
$\alpha_{x_i}$ such that for any $\alpha \in [0, \alpha_{x_i}]$, 
the $i$th inequality of (\ref{posi1}) holds, and the largest 
$\alpha_{s_i}$ such that for any $\alpha \in [0, \alpha_{s_i}]$ 
the $i$th inequality of (\ref{posi2}) holds. We then define 
\begin{eqnarray}
{\alpha^x}=\displaystyle\min_{i \in \lbrace 1,\ldots, n \rbrace}
\lbrace \alpha_{x_i}\rbrace, 
\\
{\alpha^s}=\displaystyle\min_{i \in \lbrace 1,\ldots, n \rbrace}
\lbrace \alpha_{s_i} \rbrace, 
\\
{\alpha_k}=\min \lbrace \alpha^x, \alpha^s \rbrace,
\label{alpha}
\end{eqnarray}
where $\alpha_{x_i}$ and $\alpha_{s_i}$ can be obtained, 
using a similar argument as in \cite{yang16}, in 
analytical forms represented by $\phi_k$, 
$\dot{x}_i$, $\ddot{x}_i=p_{x_i}\sigma+q_{x_i}$, 
$\psi_k$, $\dot{s}_i$, and $\ddot{s}_i=p_{s_i}\sigma+q_{s_i}$.
In every iteration, several $\sigma$ may be tried to find the
best $\sigma_k$ and $\alpha_k$ while $\phi_k$, $\psi_k$, 
$\dot{x}$, $\dot{s}$, $p_{x}$, $p_{s}$,
$q_{x}$ and $q_{s}$ are fixed (see details in the next section).
\vspace{0.1in}
%We describe a slightly more 
%efficient implementation than the one in \cite{yang16}.
%\begin{equation}
%x_i-\phi_k +\ddot{x}_i
%\ge \dot{x}_i\sin(\alpha_{x_i} )+\ddot{x}_i\cos(\alpha_{x_i} ).
%\label{alphai}
%\end{equation}
%Let $ \beta_{x_i}=\sin(\alpha_{x_i})$ and
%$ \delta_{x_i} = \cos(\alpha_{x_i})$, then we have
%$ \beta_{x_i}  = \sqrt{1-\delta_{x_i}^2}$ and
%$ \delta_{x_i}  = \sqrt{1-\beta_{x_i}^2}$. 
%Therefore, (\ref{alphai}) is equivalent to finding 
%$\beta_{x_i} \in (0,1]$ such that 
%\begin{equation}
%x_i-\phi_k +\ddot{x}_i
%\ge \dot{x}_i  \beta_{x_i} +\ddot{x}_i \delta_{x_i} .
%\label{betai}
%\end{equation}

\noindent{\it Case 1 ($\dot{x}_i=0$ and $\ddot{x}_i\ne 0$)}:

\begin{equation}
\alpha_{x_i} = \left\{
\begin{array}{ll}
\frac{\pi}{2} & \quad \mbox{if $x_i -\phi_k +\ddot{x}_i \ge 0 $} \\
\cos^{-1}\left( \frac{x_i -\phi_k +\ddot{x}_i}
{\ddot{x}_i} \right) & \quad 
\mbox{if $x_i -\phi_k +\ddot{x}_i \le 0 $}.
\end{array}
\right.
\label{case1a}
\end{equation}
\noindent{\it Case 2 ($\ddot{x}_i=0$ and $\dot{x}_i\ne 0$)}:

\begin{equation}
\alpha_{x_i} = \left\{
\begin{array}{ll}
\frac{\pi}{2} & \quad \mbox{if $\dot{x}_i \le x_i -\phi_k $} \\
\sin^{-1}\left( \frac{x_i-\phi_k }
{\dot{x}_i} \right) & \quad \mbox{if $\dot{x}_i \ge x_i -\phi_k $}
\end{array}
\right.
\label{case2a}
\end{equation}
\noindent{\it Case 3 ($\dot{x}_i>0$ and $\ddot{x}_i>0$)}:

Let 
\begin{equation}
\beta = \sin^{-1} \left( \frac{\ddot{x}_i }
{\sqrt{\dot{x}_i^2+\ddot{x}_i^2}} \right).
\label{beta1}
\end{equation}
\begin{equation}
\alpha_{x_i} = \left\{
\begin{array}{ll}
\frac{\pi}{2} & \quad \mbox{if $x_i-\phi_k + \ddot{x}_i \ge 
\sqrt{\dot{x}_i^2+\ddot{x}_i^2}$} \\
\sin^{-1}\left( \frac{x_i -\phi_k + \ddot{x}_i }
{\sqrt{\dot{x}_i^2+\ddot{x}_i^2}} \right) - \sin^{-1}\left( \frac{\ddot{x}_i } 
{\sqrt{\dot{x}_i^2+\ddot{x}_i^2}} \right) & \quad 
\mbox{if $x_i -\phi_k + \ddot{x}_i \le 
\sqrt{\dot{x}_i^2+\ddot{x}_i^2}$}
\end{array}
\right.
\label{case3a}
\end{equation}
\noindent{\it Case 4 ($\dot{x}_i>0$ and $\ddot{x}_i<0$)}:

Let 
\begin{equation}
\beta = \sin^{-1} \left( \frac{-\ddot{x}_i }
{\sqrt{\dot{x}_i^2+\ddot{x}_i^2}} \right).
\label{beta2}
\end{equation}
\begin{equation}
\alpha_{x_i} = \left\{
\begin{array}{ll}
\frac{\pi}{2} & \quad \mbox{if $x_i -\phi_k + \ddot{x}_i \ge 
\sqrt{\dot{x}_i^2+\ddot{x}_i^2}$} \\
\sin^{-1}\left( \frac{x_i-\phi_k + \ddot{x}_i }
{\sqrt{\dot{x}_i^2+\ddot{x}_i^2}} \right) + \sin^{-1}\left( \frac{-\ddot{x}_i }
{\sqrt{\dot{x}_i^2+\ddot{x}_i^2}} \right) & \quad 
\mbox{if $x_i-\phi_k + \ddot{x}_i \le 
\sqrt{\dot{x}_i^2+\ddot{x}_i^2}$}
\end{array}
\right.
\label{case4a}
\end{equation}
\noindent{\it Case 5 ($\dot{x}_i<0$ and $\ddot{x}_i<0$)}:

Let 
\begin{equation}
\beta = \sin^{-1} \left( \frac{-\ddot{x}_i }
{\sqrt{\dot{x}_i^2+\ddot{x}_i^2}} \right).
\label{beta3}
\end{equation}
\begin{equation}
\alpha_{x_i} = \left\{
\begin{array}{ll}
\frac{\pi}{2} & \quad \mbox{if $x_i-\phi_k + \ddot{x}_i \ge 0$} \\
\pi - \sin^{-1} \left( \frac{-(x_i -\phi_k + \ddot{x}_i) }
{\sqrt{\dot{x}_i^2+\ddot{x}_i^2}} \right) - \sin^{-1}\left( \frac{-\ddot{x}_i }
{\sqrt{\dot{x}_i^2+\ddot{x}_i^2}} \right) 
& \quad \mbox{if $x_i-\phi_k + \ddot{x}_i \le 0$}
\end{array}
\right.
\label{case5a}
\end{equation}
\noindent{\it Case 6 ($\dot{x}_i<0$ and $\ddot{x}_i>0$)}:

\begin{equation}
\alpha_{x_i} = \frac{\pi}{2}.
\end{equation}
\noindent{\it Case 7 ($\dot{x}_i=0$ and $\ddot{x}_i=0$)}:

\begin{equation}
\alpha_{x_i} = \frac{\pi}{2}.
\end{equation}

\noindent{\it Case 1a ($\dot{s}_i=0$, $\ddot{s}_i \ne 0$)}:

\begin{equation}
\alpha_{s_i} = \left\{
\begin{array}{ll}
\frac{\pi}{2} & \quad \mbox{if $s_i-\psi_k +\ddot{s}_i \ge 0 $} \\
\cos^{-1}\left( \frac{s_i -\psi_k +\ddot{s}_i}
{\ddot{s}_i} \right) & \quad 
\mbox{if $s_i-\psi_k +\ddot{s}_i \le 0 $}.
\end{array}
\right.
\label{case1b}
\end{equation}
\noindent{\it Case 2a ($\ddot{s}_i=0$ and $\dot{s}_i \ne 0$)}:

\begin{equation}
\alpha_{s_i} = \left\{
\begin{array}{ll}
\frac{\pi}{2} & \quad \mbox{if $\dot{s}_i \le s_i -\psi_k $} \\
\sin^{-1}\left( \frac{s_i -\psi_k }
{\dot{s}_i} \right) & \quad \mbox{if $\dot{s}_i \ge s_i -\psi_k $}
\end{array}
\right.
\label{case2b}
\end{equation}
\noindent{\it Case 3a ($\dot{s}_i>0$ and $\ddot{s}_i>0$)}:

\begin{equation}
\alpha_{s_i} = \left\{
\begin{array}{ll}
\frac{\pi}{2} & \quad \mbox{if $s_i -\psi_k + \ddot{s}_i \ge 
\sqrt{\dot{s}_i^2+\ddot{s}_i^2}$} \\
\sin^{-1}\left( \frac{s_i -\psi_k + \ddot{s}_i }
{\sqrt{\dot{s}_i^2+\ddot{s}_i^2}} \right) - \sin^{-1}\left( \frac{\ddot{s}_i } 
{\sqrt{\dot{s}_i^2+\ddot{s}_i^2}} \right) & \quad 
\mbox{if $s_i -\psi_k + \ddot{s}_i < 
\sqrt{\dot{s}_i^2+\ddot{s}_i^2}$}
\end{array}
\right.
\label{case3b}
\end{equation}
\noindent{\it Case 4a ($\dot{s}_i>0$ and $\ddot{s}_i<0$)}:

\begin{equation}
\alpha_{s_i} = \left\{
\begin{array}{ll}
\frac{\pi}{2} & \quad \mbox{if $s_i -\psi_k + \ddot{s}_i \ge 
\sqrt{\dot{s}_i^2+\ddot{s}_i^2}$} \\
\sin^{-1}\left( \frac{s_i -\psi_k + \ddot{s}_i }
{\sqrt{\dot{s}_i^2+\ddot{s}_i^2}} \right) + \sin^{-1}\left( \frac{-\ddot{s}_i }
{\sqrt{\dot{s}_i^2+\ddot{s}_i^2}} \right) & \quad 
\mbox{if $s_i-\psi_k + \ddot{s}_i \le 
\sqrt{\dot{s}_i^2+\ddot{s}_i^2}$}
\end{array}
\right.
\label{case4b}
\end{equation}
\noindent{\it Case 5a ($\dot{s}_i<0$ and $\ddot{s}_i<0$)}:

\begin{equation}
\alpha_{s_i} = \left\{
\begin{array}{ll}
\frac{\pi}{2} & \quad \mbox{if $s_i-\psi_k + \ddot{s}_i \ge 0$} \\
\pi - \sin^{-1} \left( \frac{-(s_i -\psi_k + \ddot{s}_i) }
{\sqrt{\dot{s}_i^2+\ddot{s}_i^2}} \right) - \sin^{-1}\left( \frac{-\ddot{s}_i }
{\sqrt{\dot{s}_i^2+\ddot{s}_i^2}} \right) 
& \quad \mbox{if $s_i -\psi_k + \ddot{s}_i \le 0$}
\end{array}
\right.
\label{case5b}
\end{equation}
\noindent{\it Case 6a ($\dot{s}_i<0$ and $\ddot{s}_i>0$)}:

\begin{equation}
\alpha_{s_i} = \frac{\pi}{2}.
\end{equation}
\noindent{\it Case 7a ($\dot{s}_i=0$ and $\ddot{s}_i=0$)}:

\begin{equation}
\alpha_{s_i} = \frac{\pi}{2}.
\label{case7b}
\end{equation}

\subsection{Select $\sigma_k$}

From the convergence analysis, it is clear that the best strategy 
to achieve a large step size is to select both $\alpha_k$ and
$\sigma_k$ at the same time similar to the idea of \cite{yang13a}.
Therefore, we will find a $\sigma_k$ which maximizes the step size
$\alpha_k$. The problem can be expressed as
\begin{equation}
\displaystyle\max_{\sigma \in [\sigma_{\min},\sigma_{\max}]} 
\hspace{0.05in}
\displaystyle\min_{i \in \{ 1, \ldots,n\} }
\{ \alpha_{x_i}(\sigma), \alpha_{s_i}(\sigma) \},
\label{maxmin}
\end{equation}
where $0 < \sigma_{\min} < \sigma_{\max} < 1$,
$\alpha_{x_i}(\sigma)$ and $\alpha_{s_i}(\sigma)$ are calculated 
using (\ref{case1a})-(\ref{case7b}) for a fixed 
$\sigma \in [\sigma_{\min},\sigma_{\max}]$. 
Problem (\ref{maxmin}) is a minimax problem without regularity 
conditions involving derivatives. Golden section search 
\cite{kiefer52} seems to be a reasonable method for solving this 
problem. However, given the fact from (\ref{posi1}) that 
$\alpha_{x_i}({\sigma})$ is a monotonic increasing function of 
$\sigma$ if $p_{x_i}>0$ and $\alpha_{x_i}({\sigma})$ is a 
monotonic decreasing function of $\sigma$ if $p_{x_i}<0$ 
(and similar properties hold for $\alpha_{s_i}(\sigma)$),
we can use the condition
\begin{equation}
\min \{ \displaystyle \min_{i \in p_{x_i}<0} \alpha_{x_i}(\sigma),
\displaystyle \min_{i \in p_{s_i}<0} \alpha_{s_i}(\sigma) \}
>
\min \{ \displaystyle \min_{i \in p_{x_i}>0} \alpha_{x_i}(\sigma),
\displaystyle \min_{i \in p_{s_i}>0} \alpha_{s_i}(\sigma) \},
\label{determ}
\end{equation} 
to devise an efficient bisection search to solve (\ref{maxmin}).

\begin{algorithm} {\ } \\ 
Data: $(\dot{x},\dot{s})$, $(p_x, p_s )$, $(q_x, q_s)$, $(x^k,s^k)$,
$\phi_k$, and $\psi_k$. {\ } \\
Parameter: $\epsilon \in (0,1)$, $\sigma_{lb}=\sigma_{\min}$,
$\sigma_{ub}=\sigma_{\max} \le 1$. {\ } \\
{\bf for} iteration $k=0,1,2,\ldots$
\begin{itemize}
\item[] Step 0: If $\sigma_{ub}-\sigma_{lb} \le \epsilon$, set
%$\sigma_k=\sigma$ and 
$\alpha=\displaystyle\min_{i \in \{ 1, \ldots,n\} }
\{ \alpha_{x_i}(\sigma), \alpha_{s_i}(\sigma) \}$, stop.
\item[] Step 1: Set $\sigma=\sigma_{lb}+0.5(\sigma_{ub}-\sigma_{lb})$.
\item[] Step 2: Calculate $\alpha_{x_i}(\sigma)$ and 
$\alpha_{s_i}(\sigma)$ using (\ref{case1a})-(\ref{case7b}).
\item[] Step 3: If (\ref{determ}) holds, set $\sigma_{lb}=\sigma$, 
else set $\sigma_{ub}=\sigma$.
\item[] Step 4: Set $k+1 \rightarrow k$. Go back to Step 1.
\end{itemize}
{\bf end (for)} 
\hfill \qed
\label{bisectionSearch}
\end{algorithm}

It is known that Golden section search yields a new interval 
whose length is $0.618$ of the previous interval in all iterations
\cite{luenberger84}, while the proposed algorithm yields a 
new interval whose length is $0.5$ of the previous interval, 
therefore, is more efficient.

For Algorithm \ref{mainAlgo}, after executing Algorithm 
\ref{bisectionSearch}, we may still need to further reduce
$\alpha_k$ using Golden section or bisection to satisfy
\begin{subequations}
\begin{gather}
\mu (\sigma_k, \alpha_k) < \mu_k, 
\label{objCheck}  \\
x(\sigma_k, \alpha_k) s(\sigma_k, \alpha_k) 
\ge \theta \mu(\sigma_k, \alpha_k)e.
\label{prodCheck}
\end{gather}
\end{subequations}
For Algorithm \ref{mainAlgo2}, after executing Algorithm
\ref{bisectionSearch}, we need to check only (\ref{objCheck})
and decide if further reduction of $\alpha_k$ is needed.

\begin{remark}
Comparing Lemmas \ref{positiveCond} and \ref{objDec}, we guess 
that the restriction of satisfying the condition of
$\mu (\sigma_k, \alpha_k) < \mu_k$ is weaker than the 
restriction of satisfying the conditions of
$x(\sigma_k, \alpha_k) \ge \phi_k$ and 
$s(\sigma_k, \alpha_k) \ge \psi_k$ which are solved in 
(\ref{maxmin}). Indeed, we observed that $\sigma_k$ 
and $\alpha_k$ obtained by solving (\ref{maxmin})
always satisfy the weaker restriction. But we keep this check for
the safety concern.
\end{remark}

\subsection{Rescale $\alpha_k$}

To maintain $\nu_k>0$, in each iteration, after an $\alpha_k$ is 
found as in the above process, we rescale 
$\alpha_k=\min \{0.9999 \alpha_k,0.99 \pi/2 \} < 0.99 \pi/2 $ so 
that $\nu_k=\nu_{k-1}(1-\sin(\alpha_k))>0$ holds in every iteration. 
Therefore, Assumption $4$ is always satisfied. We notice that the
rescaling also prevents $x^k$ and $s^k$ from getting too close to
the zero in early iterations, which may cause problem of solving
(\ref{useLater}). 

\subsection{Terminate criteria}

The main stopping criterion used in the implementations is 
slightly deviated from the one described in previous sections 
but follows the convention used by most infeasible interior-point software
implementations, such as LIPSOL \cite{zhang96}
\[
\frac{\|r_b^k\|}{\max \lbrace 1, \| b\| \rbrace }
+\frac{\|r_c^k \|}{\max \lbrace 1, \| c\|  \rbrace }
+\frac{ \mu_k }{\max \lbrace 1, \| c^{\T}x^k \|, \|b^{\T}\lambda^k \|  \rbrace } 
< 10^{-8}.
\] 

In case that the algorithms fail to find a good search direction,
the program also stops if step sizes $\alpha_k^x < 10^{-8}$ and
$\alpha_k^s < 10^{-8}$. 

Finally, if (a) due to the numerical problem, $r_b^{k}$ or $r_c^{k}$ does not decrease but 
$10r_b^{k-1}<r_b^{k}$ or $10r_c^{k-1}<r_c^{k}$, or (b) if $\mu<10^{-8}$, the program stops.

\section{Numerical Tests}

The two algorithms proposed in this paper are implemented 
in Matlab functions and named as {\tt arclp1.m} and 
{\tt arclp2.m}. These two algorithms are compared with two 
efficient algorithms, an arc-search algorithm proposed in 
\cite{yang16} (implemented as {\tt curvelp.m}) and the 
well-known Mehrotra's algorithm \cite{wright97,Mehrotra92}
(implemented as {\tt mehrotra.m})\footnote{The 
implementation of  {\tt curvelp.m} and {\tt mehrotra.m} are
described in detail in \cite{yang16}.} as described below. 

The main cost of the four algorithms in each iteration
is the same, involving the linear algebra for 
sparse Cholesky decomposition which is the first equation of 
(\ref{useLater}). The cost of Cholesky decomposition is  
$\mathcal O(n^3)$ which is much higher than  $\mathcal O(n^2)$,
the cost of solving the second equation of (\ref{useLater}). 
Therefore, we conclude that the iteration 
count is a good measure to compare the performance for these 
four algorithms. All Matlab codes of the above four algorithms are tested 
against to each other using the benchmark Netlib problems.
The four Matlab codes use exactly the same initial point, the 
same stopping criteria, the same pre-process and post-process 
so that the comparison of the performance of the four algorithms 
is reasonable. Numerical tests for all algorithms have been 
performed for all Netlib linear programming problems that are 
presented in standard form, except {\tt Osa\_60} ($m=10281$ 
and $n=232966$) because the PC computer used 
for the testing does not have enough memory to handle this 
problem. The iteration numbers used to solve these 
problems are listed in Table 1.

\begin{center}
\begin{longtable}{|c|c|c|c|c|}
\caption{Comparison of arclp1.m, arclp2.m, curvelp.m, and mehrotra.m for problems in Netlib}\\
\hline    
Problem  & algorithm   & iter   &  obj    & infeasibility   \\
\hline

Adlittle &  curvelp.m    & 15 &   2.2549e+05   & 1.0e-07 \\
         &  mehrotra.m   & 15 &    2.2549e+05   & 3.4e-08    \\
         &    arclp1.m   & 16  &    2.2549e+05   &  3.0e-11  \\
         &    arclp2.m   & 17 &  2.2549e+05  &   8.0e-11  \\     \hline
Afiro    &   curvelp.m   & 9  &   -464.7531    & 1.0e-11 \\
         &  mehrotra.m   &  9 & -464.7531   & 8.0e-12   \\
         &    arclp1.m   &  9 &  -464.7531   & 6.2e-13     \\
         &    arclp2.m   &  9 &  -464.7531 &  1.0e-12    \\  \hline
Agg      &   curvelp.m   & 18  &   -3.5992e+07  & 5.0e-06 \\
         &  mehrotra.m   &  22 & -3.5992e+07 & 5.2e-05  \\
         &    arclp1.m   &  20   &  -3.5992e+07   &   3.7e-06   \\
         &    arclp2.m   &  20 &   -3.5992e+07   & 7.0e-06        \\ \hline
Agg2     &   curvelp.m   & 18  &   -2.0239e+07  & 4.6e-07 \\
         &  mehrotra.m   & 20  & -2.0239e+07 & 5.2e-07  \\
         &    arclp1.m   &  21 &   -2.0239e+07   &  3.1e-08    \\
         &    arclp2.m   &  21 &  -2.0239e+07  &   2.6e-08       \\ \hline
Agg3     &   curvelp.m   & 17  &   1.0312e+07   & 3.1e-08 \\
         &  mehrotra.m   &  18 & 1.0312e+07  & 8.8e-09  \\
         &    arclp1.m   &  20 &   1.0312e+07  &   1.5e-08    \\
         &    arclp2.m   &  20 &  1.0312e+07   &   1.8e-08       \\ \hline
Bandm    &   curvelp.m   & 19  &  -158.6280    & 3.2e-11 \\
         &  mehrotra.m   &  22 &  -158.6280   & 8.3e-10  \\
         &    arclp1.m   &  20 &  -158.6280  &  3.6e-11     \\
         &    arclp2.m   &  20 &  -158.6280   &  3.4e-11 \\ \hline
Beaconfd &    curvelp.m  & 10  &  3.3592e+04  & 1.4e-12 \\
         &  mehrotra.m   &  11 &  3.3592e+04  & 1.4e-10  \\
         &    arclp1.m   &  11 &  3.3592e+04  & 1.8e-12    \\
         &    arclp2.m   &  11 &  3.3592e+04  &  6.0e-12     \\ \hline
Blend    &   curvelp.m   & 12   &   -30.8121  & 1.0e-09 \\
         &  mehrotra.m   &   14 & -30.8122    & 4.9e-11  \\
         &    arclp1.m   &   14 &  -30.8122   &   1.6e-12    \\
         &    arclp2.m   &  14  & -30.8121    &  2.5e-12       \\ \hline
Bnl1     &    curvelp.m  & 32  &   1.9776e+03  & 2.7e-09 \\
         &  mehrotra.m   &  35 &   1.9776e+03  & 3.4e-09  \\
         &    arclp1.m   & 34  &   1.9776e+03  &  2.9e-09   \\
         &    arclp2.m   & 34  &   1.9776e+03  &   7.8e-10     \\ \hline
Bnl2+    &    curvelp.m  & 31  &   1.8112e+03  & 5.4e-10 \\
         &  mehrotra.m   &  38 &   1.8112e+03  & 9.3e-07  \\
         &    arclp1.m   &  35 &   1.8112e+03  &  3.5e-06   \\
         &    arclp2.m   & 35  &   1.8112e+03  &  1.9e-07     \\ \hline
Brandy   &   curvelp.m   &  21  &  1.5185e+03  & 3.0e-06 \\
         &  mehrotra.m   &  19  &  1.5185e+03  & 6.2e-08  \\
         &    arclp1.m   &  24  &  1.5185e+03  &  2.4e-06  \\
         &    arclp2.m   &  23  &  1.5185e+03  &  1.8e-07      \\ \hline
Degen2*  &   curvelp.m   & 16   &  -1.4352e+03  & 1.9e-08 \\
         &  mehrotra.m   &   17 &  -1.4352e+03  & 2.0e-10  \\
         &    arclp1.m   &  19  &  -1.4352e+03  &  5.9e-10  \\
         &    arclp2.m   & 19   &  -1.4352e+03  &  1.5e-08    \\ \hline
Degen3*  &   curvelp.m   &  22  &  -9.8729e+02  & 7.0e-05  \\
         &  mehrotra.m   &   22 & -9.8729e+02   & 1.2e-09  \\
         &    arclp1.m   &  35  & -9.8729e+02   &  8.6e-08   \\
         &    arclp2.m   &   26 & -9.8729e+02   & 1.2e-08   \\ \hline
fffff800 &    curve      & 26  &   5.5568e+005  & 4.3e-05 \\
         &  mehrotra.m   &   31 & 5.5568e+05  & 7.7e-04  \\
         &    arclp1.m   & 28   &    5.5568e+05   &   3.7e-09  \\
         &    arclp2.m   &  28  & 5.5568e+05  &  4.9e-09     \\ \hline
Israel   &    curvelp.m  & 23  &   -8.9664e+05  & 7.4e-08 \\
         &  mehrotra.m   &   29 & -8.9665e+05 & 1.8e-08  \\
         &    arclp1.m   &  27   & -8.9664e+05   &  3.4e-08  \\
         &    arclp2.m   &  25  &  -8.9664e+05  &  6.3e-08     \\ \hline
Lotfi    &   curvelp.m   & 14  &   -25.2647     & 3.5e-10 \\
         &  mehrotra.m   &   18 & -25.2647    & 2.7e-07  \\
         &    arclp1.m   &  16  &  -25.2646 &  7.8e-09   \\
         &    arclp2.m   &  16  & -25.2647  &   6.5e-09   \\  \hline
Maros\_r7 &   curvelp.m  & 18   & 1.4972e+06   & 1.6e-08 \\
         &  mehrotra.m   & 21   &  1.4972e+06  & 6.4e-09  \\
         &    arclp1.m   &  20  &  1.4972e+06  &  1.7e-09  \\
         &    arclp2.m   &  20  &  1.4972e+06  &  1.8e-09 \\ \hline
Osa\_07* &    curvelp.m  & 37 &   5.3574e+05  & 4.2e-07 \\
         &  mehrotra.m   & 35 &   5.3578e+05  & 1.5e-07  \\
         &    arclp1.m   & 32 &   5.3578e+05  & 8.4e-10  \\
         &    arclp2.m   & 30 &   5.3578e+05  &  5.7e-5       \\ \hline
Osa\_14  &   curvelp.m   & 35 &  1.1065e+06   & 2.0e-09 \\
         &  mehrotra.m   & 37 & 1.1065e+06    & 3.0e-08  \\
         &    arclp1.m   & 42 & 1.1065e+06    & 5.2e-09   \\
         &    arclp2.m   & 42 & 1.1065e+06    &  8.9e-09   \\ \hline
Osa\_30  &  curvelp.m    & 32 &  2.1421e+06    & 1.0e-08 \\
         &  mehrotra.m   & 36 &  2.1421e+06   & 1.3e-08  \\
         &    arclp1.m   & 42 &  2.1421e+06   &  1.3e-08    \\
         &    arclp2.m   & 39 &  2.1421e+06   &  1.6e-08   \\ \hline
Qap12    &   curvelp.m   & 22 &  5.2289e+02   & 1.9e-08 \\
         &  mehrotra.m   & 24 &  5.2289e+02   & 6.2e-09   \\
         &    arclp1.m   & 23 &  5.2289e+02    & 2.9e-10  \\
         &    arclp2.m   & 22 &  5.2289e+02    &  2.5e-09   \\ \hline
Qap15*   &    curvelp.m  & 27 &  1.0411e+03    & 3.9e-07 \\
         &  mehrotra.m   & 44 &  1.0410e+03  & 1.5e-05   \\
         &    arclp1.m   & 28 &  1.0410e+03  &  8.4e-08  \\
         &    arclp2.m   & 27 &  1.0410e+03   & 1.4e-08  \\ \hline
Qap8*    &     curvelp.m & 12 &   2.0350e+02   & 1.2e-12 \\
         &  mehrotra.m   & 13 & 2.0350e+02  & 7.1e-09   \\
         &    arclp1.m   & 12 & 2.0350e+02  &  6.2e-11   \\
         &    arclp2.m   & 12 &   2.0350e+02    & 1.1e-10  \\ \hline
Sc105    &   curvelp.m   & 10  &   -52.2021     & 3.8e-12 \\
         &  mehrotra.m   &  11 & -52.2021    & 9.8e-11  \\
         &    arclp1.m   &  11 &    -52.2021   &   2.2e-12  \\
         &    arclp2.m   &  11 &  -52.2021  &   5.6e-12   \\ \hline
Sc205    &     curvelp.m & 13  &   -52.2021     & 3.7e-10 \\
         &  mehrotra.m   &  12 & -52.2021    & 8.8e-11  \\
         &    arclp1.m   & 12  &   -52.2021     &   4.4e-11 \\
         &    arclp2.m   & 12  &   -52.2021  &   4.5e-11    \\ \hline
Sc50a    &    curvelp.m  & 10  &   -64.5751     & 3.4e-12 \\
         &  mehrotra.m   &  9  & -64.5751    & 8.3e-08  \\
         &    arclp1.m   &  10 &   -64.5751    &  8.5e-13   \\
         &    arclp2.m   & 10  &  -64.5751  &    5.9e-13    \\ \hline
Sc50b    &    curvelp.m  & 8   &   -70.0000     & 1.0e-10 \\
         &  mehrotra.m   &   8 & -70.0000    & 9.1e-07  \\
         &    arclp1.m   &  10 &   -70.0000    &  3.6e-12  \\
         &    arclp2.m   & 10  &   -70.0000 &    1.8373e-12    \\ \hline
Scagr25  &  curvelp.m    & 19  &  -1.4753e+07   & 5.0e-07 \\
         &  mehrotra.m   &  18 & -1.4753e+07 & 4.6e-09  \\
         &    arclp1.m   & 19  &  -1.4753e+07  &   1.7e-08 \\
         &    arclp2.m   & 19 &   -1.4753e+07  & 2.1e-08  \\ \hline
Scagr7   &   curvelp.m   & 15  &   -2.3314e+06  & 2.7e-09 \\
         &  mehrotra.m   &  17 & -2.3314e+06 & 1.1e-07  \\
         &    arclp1.m   & 17  &   -2.3314e+06   &   7.0e-10 \\
         &    arclp2.m   &  17 &   -2.3314e+06  &  9.2e-10  \\ \hline
Scfxm1+  &    curvelp.m  & 20 &   1.8417e+04   & 3.1e-07 \\
         &  mehrotra.m   & 22 &  1.8417e+04 & 1.6e-08  \\
         &    arclp1.m   & 21 &  1.8417e+04 & 3.3e-05 \\
         &    arclp2.m   & 21 & 1.8417e+04  &  6.8e-06    \\ \hline
Scfxm2   &     curvelp.m & 23  &   3.6660e+04   & 2.3e-06 \\
         &  mehrotra.m   &   26 &  3.6660e+04 & 2.6e-08  \\
         &    arclp1.m   &  24  &   3.6660e+04   &  4.8e-05   \\
         &    arclp2.m   & 24  &  3.6661e+04 &  1.1e-05    \\ \hline
Scfxm3+  &   curvelp.m   & 24  &  5.4901e+04  & 1.9e-06 \\
         &  mehrotra.m   &  23 &  5.4901e+04  & 9.8e-08  \\
         &    arclp1.m   &  23 &  5.4901e+04  & 1.2e-04 \\
         &    arclp2.m   &  25 &  5.4901e+04  &   4.0988e-04  \\ \hline
Scrs8    &     curvelp.m & 23  &   9.0430e+02   & 1.2e-11 \\
         &  mehrotra.m   & 30  &  9.0430e+02 & 1.8e-10  \\
         &    arclp1.m   & 28  &   9.0430e+02  &  1.0e-10   \\
         &    arclp2.m   & 27  &   9.0429e+2 &   1.2e-08     \\ \hline
Scsd1    &   curvelp.m   & 12  &   8.6666       & 1.0e-10 \\
         &  mehrotra.m   &  13 &     8.6666  & 8.7e-14  \\
         &    arclp1.m   & 11  &      8.6666  & 3.3e-15  \\
         &    arclp2.m   & 11  &  8.6667  &   5.3e-15     \\ \hline
Scsd6    &   curvelp.m   & 14  &   50.5000      & 1.5e-13 \\
         &  mehrotra.m   &  16 &     50.5000 & 8.6e-15  \\
         &    arclp1.m   & 16  &   50.5000 &   2.6e-13 \\
         &    arclp2.m   & 16  &  50.5000 &  4.8e-13     \\  \hline
Scsd8    &   curvelp.m   & 13  &   9.0500e+02   & 6.7e-10 \\
         &  mehrotra.m   &  14 &  9.0500e+02    & 1.3e-10  \\
         &    arclp1.m   & 15  &  9.0500e+02    &  2.6e-13 \\
         &    arclp2.m   & 15  &  9.0500e+02    &   3.7e-13      \\ \hline
Sctap1   &   curvelp.m   & 20  &   1.4122e+03   & 2.6e-10 \\
         &  mehrotra.m   &  27 &  1.4123e+03   & 0.0031  \\
         &    arclp1.m   & 20  &  1.4123e+03   &   1.4e-11  \\
         &    arclp2.m   & 20  &   1.4123e+03  &   1.8e-11     \\ \hline
Sctap2   &   curvelp.m   & 21   &   1.7248e+03   & 2.1e-10 \\
         &  mehrotra.m   &   21 &  1.7248e+03 & 4.4e-07  \\
         &    arclp1.m   &  22  &  1.7248e+03     &  1.4e-12 \\
         &    arclp2.m   & 20   &  1.7248e+03  &   1.1e-12  \\ \hline
Sctap3   &   curvelp.m   & 20   &   1.4240e+03   & 5.7e-08 \\
         &  mehrotra.m   &   22 &  1.4240e+03 & 5.9e-07  \\
         &    arclp1.m   &  21  &   1.4240e+03  &   1.9e-12  \\
         &    arclp2.m   &  21  &   1.4240e+03  &   2.5e-12   \\ \hline
Share1b  &   curvelp.m   & 22   &   -7.6589e+04  & 6.5e-08 \\
         &  mehrotra.m   &   25 & -7.6589e+04 & 1.5e-06  \\
         &    arclp1.m   &  26  &  -7.6589e+04  &    1.9e-07 \\
         &    arclp2.m   &  26  &  -7.6589e+04  &  2.3e-07    \\ \hline
Share2b  &   curvelp.m   & 13   &   -4.1573e+02  & 4.9e-11 \\
         &  mehrotra.m   &   15 & -4.1573e+02 & 7.9e-10  \\
         &    arclp1.m   &  15  & -4.1573e+02   &   1.4e-10  \\
         &    arclp2.m   &  15  &  -4.1573e+02 &   8.9e-11  \\ \hline
Ship04l  &   curvelp.m   & 17   &   1.7933e+06   & 5.2e-11 \\
         &  mehrotra.m   &   18 &  1.7933e+06 & 2.9e-11  \\
         &    arclp1.m   &  19  &  1.7933e+06  &  1.3e-10 \\
         &    arclp2.m   &  18  &   1.7933e+06  &    5.9e-11    \\ \hline
Ship04s  &   curvelp.m   & 17   &   1.7987e+06   & 2.2e-11 \\
         &  mehrotra.m   &   20 &  1.7987e+06 & 4.5e-09  \\
         &    arclp1.m   &  19  &    1.7987e+06   & 3.1e-10 \\
         &    arclp2.m   &  19  & 1.7987e+06   &  3.1e-09   \\ \hline
Ship08l  &   curvelp.m   & 19   &   1.9090e+06   & 1.6e-07 \\
         &  mehrotra.m   &   22 &  1.9091e+06 & 1.0e-10  \\
         &    arclp1.m   &  20  &   1.9090e+06    &  1.8e-11  \\
         &    arclp2.m   & 20   &  1.9091e+06  &  1.2e-11   \\ \hline
Ship08s  &   curvelp.m   & 17   &   1.9201e+06   & 3.7e-08 \\
         &  mehrotra.m   &   20 &  1.9201e+06 & 4.5e-12  \\
         &    arclp1.m   & 19   &    1.9201e+06   &   1.7e-09 \\
         &    arclp2.m   & 19   &   1.9201e+06 &  3.2e-11  \\ \hline
Ship12l  &   curvelp.m   & 21   &   1.4702e+06   & 4.7e-13 \\
         &  mehrotra.m   &   21 &  1.4702e+06 & 1.0e-08  \\
         &    arclp1.m   & 21   &  1.4702e+06  &  3.0e-10     \\
         &    arclp2.m   & 21   &   1.4702e+06 &  6.5e-11     \\ \hline
Ship12s  &   curvelp.m   & 17   &   1.4892e+06   & 1.0e-10 \\
         &  mehrotra.m   &  19  &  1.4892e+06 & 2.1e-13  \\
         &    arclp1.m   & 21   &    1.4892e+06    & 5.0e-11 \\
         &    arclp2.m   &  20  & 1.4892e+06 &   1.4e-10    \\ \hline
Stocfor1**&  curvelp.m   &20/14 &   -4.1132e+04  & 2.8e-10 \\
         &  mehrotra.m  &   14  & -4.1132e+04 & 1.1e-10  \\
         &    arclp1.m   &  13  &   -4.1132e+04  & 8.6890e-11 \\
         &    arclp2.m   &  14  &  -4.1132e+04  &  1.1e-10   \\ \hline
Stocfor2 &   curvelp.m   & 22   &   -3.9024e+04  & 2.1e-09 \\
         &  mehrotra.m   &   22 & -3.9024e+04 & 1.6e-09  \\
         &    arclp1.m   & 22   &  -3.9024e+04    & 4.3e-09 \\
         &    arclp2.m   & 22   &   -3.9024e+04 &   4.3e-09   \\ \hline
Stocfor3 &   curvelp.m   & 34  &  -3.9976e+04   & 4.7e-08 \\
         &  mehrotra.m   & 38  & -3.9976e+04  & 6.4e-08  \\
         &    arclp1.m   & 37  & -3.9977e+04  & 7.7e-08  \\
         &    arclp2.m   & 37  & -3.9976e+04  & 6.8e-08   \\ \hline
Truss    &   curvelp.m   & 25  &  4.5882e+05    & 1.7e-07 \\
         &  mehrotra.m   & 26  &  4.5882e+05 & 9.5e-06  \\
         &    arclp1.m   & 24  &   4.5882e+05  & 5.2e-07   \\
         &    arclp2.m   & 24  &  4.5882e+05  &  1.7e-09   \\  \hline
%\label{tableIteration}
\end{longtable}
\end{center}

We noted in \cite{yang16} that {\tt curvelp.m} and
{\tt mehrotra.m} have difficulty for some problems because of
the degenerate solutions, but the proposed Algorithms 
\ref{mainAlgo} and \ref{mainAlgo2} implemented
as {\tt arclp1.m} and {\tt arclp2.m} have no difficult for all
test problems. Although we have an option of handling degenerate
solutions implemented in {\tt arclp1.m} and {\tt arclp2.m}, 
this option is not used for all the test problems.
But, {\tt curvelp.m} and {\tt mehrotra.m} have to use this
option because these two codes reached some degenerate solutions 
for several problems which make them difficult 
to solve or need significantly more iterations. 
For problems marked with '+', this option is called 
only for Mehrotra's method. For problems marked with '*', 
both {\tt curvelp.m} and {\tt mehrotra.m} need to call this 
option for better results. For problems with '**', this option 
is called for both {\tt curvelp.m} and {\tt mehrotra.m} but
{\tt curvelp.m} does not need to call this feature, however,
calling this feature reduces iteration count. We need to keep in 
mind that although using the option described in Section 
\ref{implSec}.7 reduces the iteration count significantly,
these iterations are significantly more expensive \cite{yy18}. Therefore, 
simply comparing iteration counts for problems marked with 
'+', '*', and '**'  will lead to a conclusion in favor of 
{\tt curvelp.m} and {\tt mehrotra.m} (which is what we will 
do in the following discussions).

Performance profile\footnote{To our best knowledge, performance 
profile was first used in \cite{ty96} to compare the performance
of different algorithms. The method has been becoming very 
popular after its merit was carefully analyzed in \cite{dm02}.}
is used to compare the efficiency of the four algorithms. Figure
1 is the performance profile of iteration numbers of the four
algorithms. It is clear that {\tt curvelp.m} is the most efficient
algorithm of the four algorithms; {\tt arclp2.m} is slightly 
better than {\tt arclp1.m} and {\tt mehrotra.m}; and the 
efficiencies of {\tt arclp1.m} and {\tt mehrotra.m} are roughly
the same. Overall, the efficiency difference of the four 
algorithms is not much significant. Given the fact that 
{\tt arclp1.m} and {\tt curve2p.m} are convergent in theory
and more stable in numerical test, we believe that these
two algorithms are better choices in practical applications.

\begin{figure}[ht]
\centerline{\epsfig{file=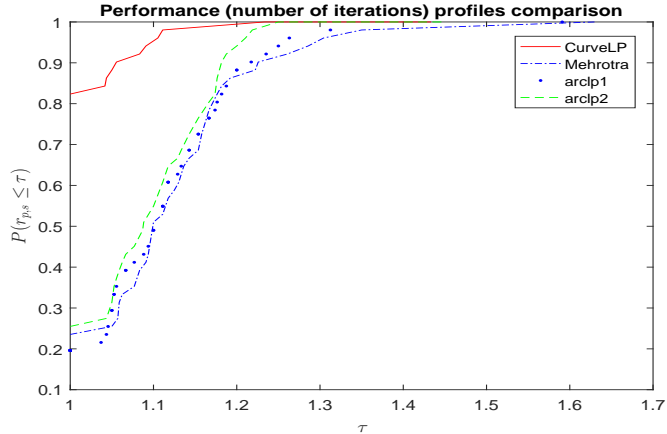,height=6cm,width=10cm}}
\caption{Performance profile comparison of the four algorithms.}
\label{fig:profile}
\end{figure}

%In view of Table 1, it is clear that Algorithm \ref{mainAlgo} 
%is as good as Mehrotra's algorithm to find the optimal solutions 
%for the tested problems. Among $51$ tested problems, Mehrotra's 
%method uses slightly fewer iterations than the newly proposed 
%method for $22$ problems, while the newly proposed method uses 
%slightly fewer iterations for $18$ problems, both methods use 
%the same number of iterations for the rest problems. The newly 
%proposed method is numerically more stable than Mehrotra's method 
%because there is only one problem, the newly proposed method need 
%to use the option described in Section~\ref{implSec}.7 and 
%Mehrotra's method does not need to use the option to solve the 
%problems; for $8$ problems, the newly proposed method does not 
%need to use the option described in Section~\ref{implSec}.7 but 
%Mehrotra's method need to use the option to solve the problems, 
%most of these problems, the iteration counts erroneously favors 
%Mehrotra's method.

\section{Conclusions}

In this paper, we propose two computationally efficient polynomial 
interior-point algorithms. These algorithms search the optimizers 
along ellipses that approximate the central path. The first
algorithm is proved to be polynomial and its simplified version
has better complexity bound than all existing infeasible 
interior-point algorithms and achieves the best complexity 
bound for all existing, feasible or infeasible, interior-point 
algorithms. Numerical test results for all the Netlib standard 
linear programming problems show that the algorithms are 
competitive to the state-of-the-art Mehrotra's Algorithm which 
has no convergence result.

\section{Acknowledgments}
  
This is a pre-print of an article published in Numerical Algorithms. 
The final authenticated version is available online at: 
https://doi.org/10.1007/s11075-018-0469-3

The author would like to thank Dr. Chris Hoxie, 
in the Office of Research at US NRC, for providing 
computational environment for this research.

%\bibliographystyle{unsrt}
%\bibliography{Myrefs}

\end{document}